\documentclass[9pt]{article}%
\usepackage{bbm}
\usepackage{epsfig}
\usepackage{graphics,graphicx,amssymb,amsmath,verbatim}
\usepackage{amssymb}
\usepackage{mathrsfs}
\usepackage{amsfonts}
\usepackage{amsmath}
\usepackage{graphicx}
\usepackage{easybmat}
\usepackage{color}%
\providecommand{\U}[1]{\protect \rule{.1in}{.1in}}
\newtheorem{theorem}{Theorem}

\newtheorem{assumption}{Assumption}

\newtheorem{definition}{Definition}

\newtheorem{lemma}{Lemma}

\newtheorem{problem}{Problem}

\newtheorem{remark}{Remark}

\newenvironment{proof}[1][Proof]{\noindent \textbf{#1.} }{\  \rule{0.5em}{0.5em}}
\definecolor{blue}{rgb}{0,0,1}
\allowdisplaybreaks[4]

\begin{document}

\title{Semi-Global Exact Prescribed-Time Stabilization of Linear Systems by Bounded Linear Time-Varying Feedback}
\author{Qinyu Xie, \quad Bin Zhou \thanks{Center for Control Theory and Guidance Technology, Harbin
Institute of Technology, Harbin, 150001, China. Email:
\texttt{binzhou@gmail.com, binzhou@hit.edu.cn}.}}

\maketitle

\begin{abstract}
This paper addresses the problem of semi-global exact prescribed-time stabilization of linear systems by bounded controls.
To solve such a problem, one needs to construct a parameterized quadratic control Lyapunov function whose corresponding level set can be made arbitrarily large (for achieving semi-global stabilization) and such that the real parts of the poles of the closed-loop system approach minus infinity as the parameter approaches infinity (for achieving prescribed-time stabilization).
{
This objective is achieved by proposing a nested parametric Lyapunov equation (PLE)-based approach, through which only two linear matrix equations need to be solved in the controller design stage without using real-time state information.
Under the proposed structural condition, the nested PLE-based approach guarantees semi-global exact prescribed-time stabilization for any prescribed bounded set of initial conditions, provided that the prescribed settling time is sufficiently large.
}
A numerical example is provided to demonstrate the effectiveness of the designed controller.
\end{abstract}

% \begin{IEEEkeywords}
% Prescribed-time stabilization;
% Semi-global stabilization;
% Time-varying feedback;
% Parametric Lyapunov equation.
% \end{IEEEkeywords}

\section{Introduction}

The primary objective of prescribed-time control is to ensure that the system state reaches and maintains the equilibrium point within a predetermined finite time (namely, finite-time stability \cite{bhat2000finite}), rather than asymptotically approaching there over an infinite duration (namely, asymptotic stability).
{
In addition to improving the convergence rate, prescribed-time control has also been studied in the presence of disturbances and uncertainties \cite{song2017time,orlov25jai}.
}
As a result, prescribed-time control has found widespread applications in advanced engineering fields that demand stringent time and accuracy requirements, for example, high-precision trajectory tracking \cite{hnlg20tc}, guidance \cite{ZZDH25TAES}, differentiator design \cite{oka21csl}, and so on.
{
Representative practical scenarios include spacecraft rendezvous under bounded thrust \cite{zhou20auto1} and prescribed-time state observation of flexible-joint robot manipulators \cite{zhou25auto}.
}

In recent years, the time-varying high-gain feedback approach has been found to be effective in solving the prescribed-time stabilization problem.
The time-varying high-gain feedback method originating from the classical fixed terminal point optimal control theory \cite{rekasius64tac} was first established in \cite{song2017time} for solving the prescribed-time stabilization problem of nonlinear systems in the normal form.
Such an approach has been extensively used for different systems and control objectives, for example, stochastic systems \cite{lk22tac,lk23tac}, multi-agent systems \cite{yws24tac}, nonlinear uncertain systems \cite{nplm23tac}, constrained nonlinear systems \cite{csx23jfi,slwc24tac}, and observer design \cite{EO24scl,holloway2019prescribed}.
Very recently, linear time-varying high-gain feedback designed from solutions to a class of parametric Lyapunov equations (PLEs) was used to solve the prescribed-time stabilization problem for general linear systems \cite{zhou20auto1,zhou20auto2} and some nonlinear systems \cite{YS25auto,zs21tac}.
{
Such controllers generally employ a time-varying gain that becomes unbounded as the prescribed settling time is approached.
For instance, the PLE-based controller in \cite{zhou20auto1} relies on a time-varying gain parameter that increases without bound when approaching the prescribed time.
At the prescribed settling time, this gain becomes singular and the corresponding feedback matrix obtained from the PLE is no longer well defined.
Moreover, a formal extension of this gain beyond the prescribed time results in parameters outside the admissible range required for the existence of the positive-definite PLE solution.
Therefore, such time-varying high-gain feedback controllers cannot be directly extended to or beyond the prescribed settling time.
Recently, the so-called time deformation approach has been established to resolve this problem \cite{orlov22auto}, which also allows certain perturbations to be fully rejected.
}

High-gain feedback will necessarily require large control effort,
{especially during the initial stage, when the state is far from the origin}, which will certainly cause actuator saturation.
The input constraints must be considered in practice since otherwise it will cause performance degradation and even instability \cite{teel92scl,hjs02tac}.
To the best of our knowledge, in most existing literature, input constraints are typically addressed qualitatively.
The impact of input constraints on the size of the domain of attraction was quantitatively studied only in a few works \cite{zhou20auto1, zhou20auto2}.
It was established in \cite{zhou20auto1} that semi-global prescribed-time stabilization of linear systems with constrained (bounded) controls is achievable if all the eigenvalues of the open-loop system are on the imaginary axis.

{
In this paper, we extend the semi-global exact prescribed-time stabilization result in \cite{zhou20auto1} to a broader class of linear systems with bounded controls, where the open-loop system is allowed to have stable poles in addition to poles on the imaginary axis.
More specifically, the following semi-global solvability property holds if and only if the system is controllable and all open-loop poles lie in the closed left-half plane.
For every prescribed bounded set of initial conditions, the stabilization problem is solvable for all sufficiently large prescribed settling times.
The necessity of this structural condition for the above semi-global solvability property follows from classical results on constrained null controllability \cite{sb80siam}: null controllability from arbitrary initial states under bounded controls requires controllability and excludes open right-half-plane eigenvalues.
The main difficulty in establishing the sufficiency lies in retaining the semi-global property under the input constraints, namely, ensuring that any prescribed bounded set of initial conditions, no matter how large, can be included in the domain of attraction through a suitable choice of the controller parameters.
In \cite{zhou20auto1}, this property is achieved by making the parameterized matrix associated with the PLE arbitrarily small.
However, when the open-loop system contains stable poles, the parameter in the PLE of \cite{zhou20auto1} is constrained by a strictly positive lower bound.
As a consequence, the associated matrix can no longer be made arbitrarily small through the mechanism used in \cite{zhou20auto1}, and the semi-global result cannot be directly extended to this more general class of systems.}

{
To overcome this difficulty, we propose a new nested PLE-based approach.
The key feature of the proposed construction is that the resulting parameterized positive definite matrix can still be made arbitrarily small even when the open-loop system contains stable poles.
This restores the essential mechanism for enlarging the domain of attraction arbitrarily under bounded controls.
Meanwhile, the constructed matrix possesses the additional high-gain and differential properties required for exact prescribed-time convergence and the corresponding stability analysis.
Based on these properties, a linear time-varying high-gain feedback controller is developed to establish the sufficiency part of the above structural characterization, showing that the stabilization problem is solvable for every prescribed bounded set of initial conditions and all sufficiently large prescribed settling times.
Therefore, compared with \cite{zhou20auto1}, the main technical advance lies in the nested PLEs construction, which preserves the key property required for semi-global stabilization in the presence of stable open-loop poles.
}

\textbf{Notation:}
For a complex number $ c $, we use $ \Re(c) $ to denote its real part.
For a square matrix  $ A $, we use  $ A^{\mathrm{T}} $,  $ \|A\| $, $ \mathrm{tr}(A) $, $ \lambda(A) $, $ \lambda_{i}(A) $, $ \det(A) $, and $ \mathrm{adj}(A) $
to denote its transpose, 2-norm, trace, eigenvalue set, the $i$-th eigenvalue, determinant, and {adjugate matrix}, respectively.
Denote $ \phi(A)=\min_{i=1,2,\ldots,n}\{ \Re(\lambda_{i}(A)) \} $ and $ \alpha(A)=\max_{i=1,2,\ldots,n}\{ \Re(\lambda_{i}(A)) \} $.
For a symmetric matrix  $ P $, we use  $ \lambda_{\max}(P) $  and  $ \lambda_{\min}(P) $ to denote respectively its maximal and minimal eigenvalues.
Moreover, if  $ P\ge 0 $, the square root of  $ P $  is denoted as  $ P^{\frac{1}{2}} $.
The standard saturation function $ \sigma: \mathbf{R}^{m} \rightarrow [-1,1]^{m} $ is defined as $ \sigma(u)=  [ \allowbreak \operatorname{sign}(u_{1})\min\{ |u_{1}|,1 \},\allowbreak \ldots, \operatorname{sign}(u_{m})\min\{ |u_{m}|,1 \} ]^{\mathrm{T}} $.
 Finally, the class $K$ function is denoted as  $ \mathcal{K} $.

\section{Problem Formulation and Preliminaries}

\subsection{Problem Formulation}

To state our problem to be solved, we first give the following definition from \cite{zhou25auto}.

\begin{definition}
\label{defefts}The nonlinear time-varying system
\begin{equation}
\dot{x}=f(t,x),\quad x\in \mathbf{R}^{n}, \label{systemnon}
\end{equation}
with $f(t,0)=0$ is said to be $T$-finite-time stable ($T$-FTS) with $\mathit{\Omega}$ being contained in the domain of attraction if it is stable in the
sense of Lyapunov, and $x(0) \in \mathit{\Omega} \Rightarrow \lim_{t\uparrow T} \Vert x(t) \Vert =0$.
If, in addition, $x(t)\neq0,\forall t\in [0,T)$ as long as $x(0)\neq0$, then it is said to be exact $T$-FTS.
\end{definition}

The notion of $T$-FTS in Definition~\ref{defefts} is closely related to the prescribed-time regulation in \cite{song2017time} and the predefined-time stability in \cite{sanchez18ima}.
In \cite{song2017time}, the convergence time is prescribed in advance, but the corresponding stability notions do not require a nonzero trajectory to remain nonzero before the prescribed terminal time.
For predefined-time stability \cite{sanchez18ima}, the predefined time $T_c$ is a uniform upper bound on the settling-time function; in its strong version, $T_c$ is the least upper bound, i.e., $\sup_{x_0}T(x_0)=T_c$, while individual trajectories may settle earlier.
In contrast, exact $T$-FTS additionally requires that every nonzero trajectory in the considered domain remain nonzero for all $t<T$ and converge to zero as $t\uparrow T$.
Thus, the prescribed time $T$ is the exact convergence instant for every nonzero initial condition in the considered domain.

{
Consider the linear system
\begin{equation}
\dot{x}=Ax+Bu, \label{system}
\end{equation}
where $ x\in\mathbf{R}^{n} $ is the state and
$ u \in \mathbf{R}^{m} $ is the control input.
}
The problem to be solved for the above system is formulated as follows.

\begin{problem} \label{prob1} (semi-global exact prescribed-time stabilization)
Let $ \mathit{\Omega}\subset\mathbf{R}^{n} $ be a prescribed set that can be arbitrarily large yet bounded, and $ T>0 $ be a prescribed constant.
{
Find a controller $U(t,x(t))$, $0\le t<T$, such that, with the saturated control input $ u(t)=\sigma(U(t,x(t))) $, the closed-loop system is exact $T$-FTS in the sense of Definition \ref{defefts}, with $\mathit{\Omega}$ being contained in the domain of attraction.
}
\end{problem}

\subsection{An Assumption and Discussions}

In order to solve Problem \ref{prob1}, the following assumption is imposed on the system.
\begin{assumption} \label{ass1}
    $(A,B)$ is controllable and all eigenvalues of  $ A $  are located in the closed left-half plane.
    \label{assump1}
\end{assumption}

The above assumption is necessary.
On the one hand, if $ (A,B) $ is only stabilizable and $ A $ has some stable eigenvalues that are uncontrollable, then the state trajectories corresponding to such a stable eigenvalue will converge to zero exponentially no matter how the controller is designed.
On the other hand, if $A$ contains eigenvalues having positive real part, then the system is only locally stabilizable by any feedback (see, for example, \cite{hjs02tac,sb80siam,teel92scl}).

Problem \ref{prob1} is not always solvable even if Assumption \ref{ass1} is fulfilled.
This is because, in the presence of input constraints, there is a trade-off between the settling time and the size of the domain of attraction, say, the larger the size of the domain of attraction, the longer the settling time, which is reasonable.
To make this clear, we present the following simple lemma whose proof is given in Appendix A.

\begin{lemma} \label{lm1}
Consider the following scalar linear system
\begin{equation}
    \dot{x}=- \alpha_{0} x+u,\ |u|\le 1,\; x \in \mathbf{R}, \;u \in \mathbf{R}, \;x(0)=x_{0}, \label{sys2}
\end{equation}
where $ \alpha_{0} \geq 0 $ is a constant.
Let  $ \mathit{\Omega}=\{ x\ |\ \vert x \vert \le D \} $ where $ D $ is a positive constant, and $T^{\ast}(\mathit{\Omega})=\inf\{ T:\ \forall x_{0} \in \mathit{\Omega},\ \exists u\in \mathbf{R}, \text{such that} \; x(t)=0,\ \forall t \ge T \}$.
Then
\begin{equation}
T^{\ast }(\mathit{\Omega })=\left \{
\begin{array}{ll}
\frac{1}{\alpha_{0} }\ln (1+\alpha_{0} D), & \alpha_{0} >0, \\
D, & \alpha_{0} =0.%
\end{array}%
\right. \label{eqzb2}
\end{equation}
\end{lemma}

It follows from the above lemma that Problem \ref{prob1} is solvable for system (\ref{sys2}) if and only if $ T \geq T^{*}(\mathit{\Omega}) $.
{
For a general system, however, an exact characterization of the minimum feasible settling time is not pursued in this paper.
Instead, our objective is to characterize the system class for which, for every prescribed bounded set $\mathit{\Omega}$, Problem~\ref{prob1} is guaranteed to be solvable for all sufficiently large prescribed settling times, and to provide constructive sufficient bounds on such settling times.
Moreover, $T^{*}(\mathit{\Omega})$ is strictly increasing with respect to the radius $D$ of the prescribed domain, since $\mathrm{d}T^{*}/\mathrm{d}D=1/(1+\alpha_{0}D)>0$ for $\alpha_{0}>0$, whereas $\mathrm{d}T^{*}/\mathrm{d}D=1$ for $\alpha_{0}=0$.
Hence, enlarging the prescribed domain necessarily increases the minimum admissible settling time under the bounded input.
}

\subsection{Existing Solution}

Problem \ref{prob1} was solved in our early paper \cite{zhou20auto1} under the following assumption:
\begin{assumption} \label{ass2}
    $ (A,B) $ is controllable and all eigenvalues of  $ A $  are on the imaginary axis.
\end{assumption}

The main difference between Assumption \ref{ass1} and Assumption \ref{ass2} is that the former one allows $A$ to have stable eigenvalues (that are controllable).
Under Assumption \ref{ass2}, the controller designed in \cite{zhou20auto1} is a linear time-varying high-gain feedback in the form of
\begin{equation*}
u(t)=-B^{\mathrm{T}}P_{0}(\gamma(t))x(t), \quad t \in [0,T), \label{equ1}
\end{equation*}
where $\gamma(t)=T\gamma_{0}/(T-t)$, $\gamma_{0}$ is a sufficiently small number so that the domain of attraction of the closed-loop system contains $\mathit{\Omega}$, $T>T^{\ast}$, where $T^{\ast}$ is some constant depending on $\mathit{\Omega}$, and $P_{0}$ satisfies the following PLE
\begin{equation}
    A^{\mathrm{T}}P_{0}+P_{0}A-P_{0}BB^{\mathrm{T}}P_{0}=-\gamma P_{0}.
    \label{PLE}
\end{equation}
Such a result heavily relies on properties of the solution to the above PLE, as recalled in the following lemma.

\begin{lemma} \cite{zhou20auto1} \label{lm2}
Let $(A,B)$ be controllable.
Then the PLE (\ref{PLE}) has a unique positive definite solution if and only if
\begin{equation}
\gamma>-2\phi(A),
\label{eqgamma}
\end{equation}
and the unique solution is given by $P_{0}(\gamma)=W^{-1}(\gamma)$ where $W(\gamma)$ is a rational matrix function satisfying
\begin{equation*}
     \left( A+\frac{\gamma}{2}I_{n} \right) W+W\left( A+\frac{\gamma}{2}I_{n} \right) ^{\mathrm{T}}=BB^{\mathrm{T}}.
\end{equation*} Moreover, $\lambda_{i}(A-BB^{\mathrm{T}}P_{0})=-\lambda_{i}(A)-\gamma,i=1,2,\ldots,n$.
If in addition all eigenvalues of $A$ are on the imaginary axis, then $\lim_{\gamma \downarrow 0} P_{0}(\gamma)=0$.
\end{lemma}

The solution proposed in \cite{zhou20auto1} relies on essentially finding a control Lyapunov function $x^{\mathrm{T}}P(\gamma)x$, parameterized by $\gamma=\gamma(t)$ which is an increasing function of $t$ and {approaches infinity as $t$ approaches $T$}, such that
\begin{equation}
\lim _{\gamma  \downarrow 0} P(\gamma )\!\!=\!\!0, \lim _{\gamma \rightarrow \infty} \Re \{\lambda_{i}(A\!\!-\!\!BB^{\mathrm{T}}P(\gamma))\}\!\!=\!\!-\infty.
\label{eqtwocond}
\end{equation}
The first equation in (\ref{eqtwocond}) guarantees that the domain of attraction can be made arbitrarily large by decreasing $\gamma_{0}=\gamma(0)$, and the second equation in (\ref{eqtwocond}) guarantees the $T$-FTS by allowing $\gamma=\gamma(t)$ to {approach infinity as $t$ approaches $T$}.
It follows from Lemma \ref{lm2} that equation (\ref{eqtwocond}) is satisfied with $P=P_{0}$ under Assumption \ref{ass2}.

Unfortunately, equation (\ref{eqtwocond}) cannot be satisfied with $P=P_{0}$ if Assumption \ref{ass1} is fulfilled but Assumption \ref{ass2} is not (namely, $A$ has at least one stable eigenvalue), since, by virtue of (\ref{eqgamma}), the parameter $\gamma$ should be larger than the positive constant $-2\phi(A)$, and thus cannot satisfy the first equation in (\ref{eqtwocond}), namely, it cannot be made arbitrarily small so that the domain of attraction of the closed-loop system contains the prescribed set $\mathit{\Omega}$.

% In this paper, we will present a complete solution to Problem \ref{prob1} under Assumption \ref{ass1}.
{
In this paper, we will show that, under Assumption~\ref{ass1}, for every prescribed bounded set $\mathit{\Omega}$ there exists a finite settling-time threshold such that Problem~\ref{prob1} admits a constructive solution for every prescribed time above this threshold.
}
To this end, we will establish a new nested PLE-based approach to construct a control Lyapunov function $x^{\mathrm{T}}P(\gamma)x$ such that these two conditions in equation (\ref{eqtwocond}) are satisfied.

\section{Main Result}

\subsection{Nested PLE-Based Approach}

In this subsection, we present a nested PLE-based approach to construct a positive definite matrix $P(\gamma)$ parameterized by $\gamma$ such that these two conditions in equation (\ref{eqtwocond}) are satisfied, and, at the same time, establish some further properties of $P(\gamma)$ that will also be very important in the stability analysis of the closed-loop system.
{
The two PLEs are called ``nested'' because the solution of the first one is used as a known quantity in the second one. Hence, they are solved sequentially rather than simultaneously as mutually coupled equations.
}
Hereafter, we omit the dependence of $ \gamma $  in variables, for example, $ P=P(\gamma) $.

\begin{theorem} \label{theoP}
Assume that $(A,B)$ satisfies Assumption \ref{ass1}.
Consider the following two nested PLEs
\begin{gather}
    A^{\mathrm{T}}X+XA+XBB^{\mathrm{T}}X= \gamma X,   \label{ple0}\\
    ( A\!\!+\!\!BB^{\mathrm{T}}X)^{\mathrm{T}}Y\!\!+\!\!Y( A\!\!+\!\!BB^{\mathrm{T}}X)\!\!-\!\!YBB^{\mathrm{T}}Y\!\!=\!\!- \gamma Y,  \label{ple1}
\end{gather}
and define $P=Y-X$. Then the following holds.
\begin{enumerate}
\item The PLE \eqref{ple0} and \eqref{ple1} have a unique positive definite solution $ X $ and $ Y $ for any $\gamma >0$, respectively,
and $P$ is a rational matrix function of $\gamma$ and such that
\begin{equation}
\lambda_{i}(A-BB^{\mathrm{T}}P)=\lambda_{i}(A)-2\gamma,i=1,2,\ldots,n.  \label{poles}
\end{equation}
\item $P$  is positive definite for any $\gamma > 0 $.
\item $\mathrm{tr}( B^{\mathrm{T}}PB)=2n\gamma$.
\item There exists a $k>0$ such that the following Lyapunov inequality holds for any $\gamma >0$:
\begin{equation}
    -k (\gamma-\phi(A)) P \!\le\! A^{\mathrm{T}}P\!+\!PA\!-\!PBB^{\mathrm{T}}P \!\le \!- \gamma P. \label{pli}
\end{equation}
\item There exist positive constants $ c_{1} $  and  $ c_{2} $  such that
\begin{equation*}
   \lim_{\gamma \to \infty}\frac{\lambda_{\max}(P)}{\gamma^{2\mu_{\mathrm{c}}-1}}=c_{1},\; \lim_{\gamma \downarrow 0}\frac{\lambda_{\max}(P)}{\gamma}=c_{2},\label{eigmaP}
\end{equation*}
where $ \mu_{\mathrm{c}} $ is the controllability index of $(A,B)$.

{
\item There exist a constant $\delta_{\mathrm c}>0$ and a positive scalar function $f(\gamma)$ such that
\begin{equation}
    \frac{P}{2n\gamma}
    <
    \frac{\mathrm{d}P}{\mathrm{d}\gamma}
    \leq
    \frac{f(\gamma)P}{2n\gamma}
    \leq
    \frac{\delta_{\mathrm{c}}P}{2n\gamma},
    \quad \forall\,\gamma>0,
    \label{defdc}
\end{equation}
where $ f(\gamma)=\min\{ \delta_{\mathrm{c}},\overline{\delta}_{\mathrm{c}} \} $,
\begin{equation*}
    \overline{\delta}_{\mathrm{c}}=
    \frac{1}{n}\mathrm{tr}\big(H\big) \!+\!
    \sqrt{ \frac{n-1}{n} \left( \mathrm{tr}\big(H^{2}\big) - \frac{1}{n} \big(\mathrm{tr}(H)\big)^{2}
    \right)},
\end{equation*}
with $ H=2n\gamma P^{-\frac{1}{2}}\frac{\mathrm{d}P}{\mathrm{d}\gamma}P^{-\frac{1}{2}}$.
}

\end{enumerate}
\end{theorem}
The proof of Theorem \ref{theoP} is quite complicated, and is thus moved to Appendix B.
It follows from Item 5 and Item 2 of Theorem \ref{theoP} that the matrix $P$ satisfies (\ref{eqtwocond}), and is thus a suitable control Lyapunov function, based on which a solution to Problem \ref{prob1} will be established in the next subsection. Notice that the two PLEs (\ref{ple0}) and (\ref{ple1}) are respectively equivalent to
\begin{align}
    \tilde{X}\left( A-\frac{\gamma}{2}I_{n} \right)^{\mathrm{T}} +\left( A-\frac{\gamma}{2}I_{n} \right)\tilde{X} =-BB^{\mathrm{T}}, \label{lyap0}\\
    \tilde{Y}\!\left(\! A\!\!+\!\!BB^{\mathrm{T}}X\!\!+\!\!\frac{\gamma}{2}I_{n}\!\right)^{\mathrm{T}}\!\!+\!\!\left( \! A\!\!+\!\!BB^{\mathrm{T}}X\!\!+\!\!\frac{\gamma}{2}I_{n}\!\right)\!\tilde{Y}\!\!=\!\!BB^{\mathrm{T}}\!,\label{lyap00}
    \end{align}
where $\tilde{X}=X^{-1}$ and $\tilde{Y}=Y^{-1}$.
{
The PLEs \eqref{ple0} and \eqref{ple1} have Riccati-type algebraic forms due to their quadratic matrix terms, but they are algebraic matrix equations parameterized by the scalar $\gamma$, rather than time-varying Riccati differential equations.
By introducing $\tilde{X}=X^{-1}$ and $\tilde{Y}=Y^{-1}$, they can be transformed into the linear matrix equations \eqref{lyap0} and \eqref{lyap00}, respectively.
In our implementation, these parameterized equations are handled analytically in the design stage.
Let $F_X(\gamma)=A-\frac{\gamma}{2}I_n$.
Vectorizing \eqref{lyap0} yields $(I_n\otimes F_X(\gamma)+F_X(\gamma)\otimes I_n) \operatorname{vec}(\tilde X)=-\operatorname{vec}(BB^{\mathrm T})$, from which $\tilde X(\gamma)$ and hence $X(\gamma)$ are obtained.
After defining $F_Y(\gamma)=A+BB^{\mathrm T}X(\gamma)+\frac{\gamma}{2}I_n$, vectorizing \eqref{lyap00} similarly gives $(I_n\otimes F_Y(\gamma)+F_Y(\gamma)\otimes I_n) \operatorname{vec}(\tilde Y)=\operatorname{vec}(BB^{\mathrm T})$, yielding $Y(\gamma)$.
Thus, \(P(\gamma)=Y(\gamma)-X(\gamma)\) is obtained analytically as an explicit rational matrix function of \(\gamma\).
As guaranteed by Theorem 1, \(X(\gamma)\),  \(Y(\gamma)\), and \(P(\gamma)\) are positive definite and nonsingular for every \(\gamma>0\) in exact arithmetic.
During closed-loop operation, only the prederived expression \(P(\gamma(t))\) needs to be evaluated.
To reduce finite-precision loss for small \(\gamma\), \(P(\gamma)\) can be simplified analytically and evaluated directly, rather than formed by subtracting separately evaluated \(X(\gamma)\) and \(Y(\gamma)\).
}

\subsection{Solution to Problem \ref{prob1}}

In this subsection, we will present a solution to Problem \ref{prob1} based on Theorem \ref{theoP}.
To give our main result, we also need some preliminaries.
A piecewise-continuous function $\lambda(t):[0,T)\rightarrow \mathbf{R}$ is said to be $T$-FTS if
$\lim_{t\uparrow T}\lambda(t)=-\infty.$
The following lemma for testing the exact  $T$-FTS is recalled.

\begin{lemma} \label{lemmaefts}
Let  $ T>0 $  be a given constant.
Suppose that there is a function  $ V(t,x):[0,T) \times\mathbf{R}^{n}
\rightarrow \mathbf{R} $  such that, for any  $ t \in [0,T) $  and  $ x \in \mathbf{R}^{n} $,
   \begin{gather}
       \kappa_{1}(t)v_{1}(\| x \|)\le V(t,x)\le \kappa_{2}(t)v_{2}(\| x \|), \label{eqzb25}\\
       \dot{V}(t,x)|_{\eqref{systemnon}}\ge  \nu(t)V(t,x), \label{eqzb26}\\
           V(t,x)\leq c \Rightarrow \dot{V}(t,x)|_{\eqref{systemnon}}\le \mu(t)V(t,x), \label{eqzb27}
   \end{gather}
where  $ v_{1},v_{2} \in \mathcal{K} $,  $ \kappa_{1},\kappa_{2}:[0,T)\rightarrow (0,\infty) $,
$ \nu:[0,T)\rightarrow \mathbf{R} $ and $ \mu:[0,T)\rightarrow (-\infty,0) $ are piecewise continuous functions.
If
\begin{equation*}
    \rho(t)=\int_{0}^{t}\mu(s)\mathrm{d}s -\ln(\kappa_{1}(t)),
\end{equation*}
is a  $T$-FTS function, then system \eqref{systemnon} is exact $ T $-FTS with $\{x\; |\; V(0,x)\leq c \}$ being contained in the domain of attraction, and the state  $ x $  satisfies, for any  $ t \in [0,T) $,
\begin{gather*}
    v_{1}(\| x(t) \|)\le \frac{\kappa_{2}(0)}{\kappa_{1}(t)}\exp\left( \int_{0}^{t} \mu(s)\mathrm{d}s \right) v_{2}(\| x(0) \|),\\
    v_{2}(\| x(t) \|)\ge \frac{\kappa_{1}(0)}{\kappa_{2}(t)}\exp\left( \int_{0}^{t} \nu(s)\mathrm{d}s \right)v_{1}(\| x(0) \|).
\end{gather*}
\end{lemma}

The above lemma is a local version of Lemma 1 in \cite{zhou25auto}, and can be proven in a similar way. The proof is thus omitted to save space.

Let $P$ be determined in Theorem \ref{theoP} and define the ellipsoid
\begin{equation}
    \mathscr{E}(\gamma) = \{ x \in \mathbf{R}^{n} \ | \ 2n\gamma x^{\mathrm{T}} P x \le 4\}. \label{defEs}
\end{equation}
\begin{theorem}
Assume that system \eqref{system} satisfies Assumption \ref{assump1}. Let $\gamma_0 >0$ be such that
$\mathit{\Omega}\subseteq \mathscr{E}(\gamma_{0}) $ where $P=P(\gamma)$ is determined in Theorem \ref{theoP}.
Then, for any prescribed time $T$ satisfying
\begin{equation}
    T > \frac{2n+\delta_{\mathrm{c}}}{2n\gamma_{0}}, \label{defT}
\end{equation}
the following linear time-varying state feedback defined on  $ t \in [0,T) $
{
\begin{equation}
    U(t,x)=-\left( \frac{1}{2}+\eta \right)  B^{\mathrm{T}}P(\gamma(t))x,\  \gamma(t)=\frac{T}{T-t}\gamma_{0},
\label{defgamat}
\end{equation}
}
is a solution to Problem \ref{prob1}, where $ \eta \geq 0 $ is any constant.
Moreover, for any  $ x(0) \in \mathscr{E}(\gamma_{0}) $, the state and control for the closed-loop system satisfy
\begin{gather}
   \alpha_{1} (T-t)^{\alpha_{3}} \| x(0) \| \le \|x(t)\| \le  \alpha_{2} (T-t)^{\alpha_{4}} \| x(0) \| ,\label{xle}\\
    \| u(t) \| \le  \alpha_{5} (T-t)^{\alpha_{4}}\| x(0) \|,\label{ule}
\end{gather}
where  $ \alpha_{i}>0,\ i=1,2,\ldots,5 $  are constants.
\label{theostate}
\end{theorem}

\begin{proof}
First, we notice that, as $(A,B)$ satisfies Assumption \ref{ass1}, it follows from Item 5 of Theorem \ref{theoP} that $\lim_{\gamma\downarrow 0} P(\gamma)=0$, which, in view of the fact that $\mathit{\Omega}$ is bounded, implies that there always exists a $\gamma_{0}>0$ such that $\mathit{\Omega}\subseteq \mathscr{E}(\gamma_{0}) $ is satisfied.

For simplicity, we omit the dependence of variables on $t$ in the following.
The closed-loop system consisting of \eqref{system} and \eqref{defgamat} is
\begin{equation}
    \dot{x}=Ax-B\sigma\left( \left( \frac{1}{2}+\eta \right)B^{\mathrm{T}}Px  \right)  .    \label{systemclosestate}
\end{equation}
Consider the Lyapunov-like function as $V=V(t,x(t))=2n\gamma x^{\mathrm{T}}(t)P(\gamma(t))x(t),t\in [0,T).$

According to the definition of  $ V $, for any  $ t \in [0,T) $, we can obtain from (\ref{defdc}) that $V \ge 2n\lambda_{\min}(P(\gamma_{0})) \gamma_{0}\| x \|^{2}.$
According to Item 5 of Theorem \ref{theoP}, there exists a constant $ \varphi(\gamma_{0}) $ depending on $ \gamma_{0} $ such that $\| P (\gamma) \|\le \varphi(\gamma_{0})\gamma^{2\mu_{\mathrm{c}}-1},\forall \gamma\ge \gamma_{0} $, from which it follows that $ V\leq 2n\varphi(\gamma_{0}) \gamma^{2\mu_{\mathrm{c}}}  \| x \|^{2} $.
Therefore, $ V(t,x) $ satisfies condition (\ref{eqzb25}) in Lemma \ref{lemmaefts} with $ v_{1}(s)=v_{2}(s)=s^{2} $, $ \kappa_{1}=2n\lambda_{\min}(P(\gamma_{0})) \gamma_{0} $ and $ \kappa_{2}(t)=2n\varphi(\gamma_{0}) \gamma^{2\mu_{\mathrm{c}}}(t)$.

By Item 3 of Theorem \ref{theoP} we have
\begin{equation}
PBB^{\mathrm{T}}P\leq P^{\frac{1}{2}}\mathrm{tr}\left(P^{\frac{1}{2}}BB^{\mathrm{T}}P^{\frac{1}{2}}\right)P^{\frac{1}{2}}=2n\gamma P.
\label{eqdVleV0}
\end{equation}
Denote $ B=[b_{1}\ b_{ 2}\ \cdots \ b_{m}]$. Based on Items 4 and 6 of Theorem \ref{theoP} and (\ref{eqdVleV0}), the time derivative of $V$ along the trajectories of system \eqref{systemclosestate} can be evaluated as
\begin{align*}
   \dot{V}=&
    2n\dot{\gamma}x^{\mathrm{T}}Px+ 2n\gamma  x^{\mathrm{T}}\dot{P} x + 2n\gamma  x^{\mathrm{T}}(A^{\mathrm{T}}P+PA)x  \\
    & -4n \gamma  \sum_{i=1}^{m}x^{\mathrm{T}}Pb_{i}\sigma\left( \left( \frac{1}{2}+\eta \right)b_{i}^{\mathrm{T}}Px  \right)\\
    \ge  & 2n\dot{\gamma}x^{\mathrm{T}}Px+ 2n\gamma  x^{\mathrm{T}}\dot{P} x + 2n\gamma  x^{\mathrm{T}}PBB^{\mathrm{T}}Px  \\
    & -2nk\gamma(\gamma-\phi(A)) x^{\mathrm{T}}Px\\
    & -4n \gamma \left( \frac{1}{2}+\eta \right)  \sum_{i=1}^{m}x^{\mathrm{T}}Pb_{i} b_{i}^{\mathrm{T}}Px \\
    \ge & -2nk\gamma ( \gamma-\phi(A) )   x^{\mathrm{T}}Px- 4\eta n\gamma  x^{\mathrm{T}}PBB^{\mathrm{T}}Px\\
    \ge & -2nk\gamma ( \gamma-\phi(A) )   x^{\mathrm{T}}Px- 8\eta n^{2}\gamma^{2}  x^{\mathrm{T}}Px\\
    = & -\left( k\frac{ \gamma-\phi(A) }{\gamma} +4\eta n   \right) \gamma V\\
    \ge& -\left( 4\eta n+k-\frac{k\phi(A)}{\gamma_{0}} \right) \gamma V\\
    = & -\varepsilon_{1} \gamma V, \quad \forall t \in [0,T),
\end{align*}
where $\varepsilon_{1}=4 \eta n+k-k\phi(A)/\gamma_{0}{>0}$ is a  constant.
Therefore, condition (\ref{eqzb26}) in Lemma \ref{lemmaefts} is satisfied with $ \nu(t)=-\varepsilon_{1}\gamma(t) $.

{
To verify condition (\ref{eqzb27}) in Lemma \ref{lemmaefts} with $c=4$, consider any $(t,x)$ satisfying $ V(t,x)=2n\gamma x^{\mathrm{T}}Px \le 4 $.
With the help of Item 3 of Theorem \ref{theoP}, we have
\begin{align*}
    \left| \frac{1}{2} b_{i}^{\mathrm{T}}Px\right|^{2}
    &=\frac{1}{4}x^{\mathrm{T}}Pb_{i}b_{i}^{\mathrm{T}}Px\\
    &\leq
    \frac{1}{4}\mathrm{tr}(B^{\mathrm{T}}PB)x^{\mathrm{T}}Px\\
    &=\frac{n}{2}\gamma x^{\mathrm{T}}Px
    =\frac{1}{4}V(t,x)
    \leq 1 .
\end{align*}
Hence, $\left|\frac{1}{2}b_{i}^{\mathrm{T}}Px\right|\le1$.
By the definition of the standard saturation function, $ \sigma(s)=s $ for $ |s| \le 1 $, and therefore $ \sigma\left(\frac{1}{2}b_{i}^{\mathrm{T}}Px\right) = \frac{1}{2}b_{i}^{\mathrm{T}}Px$.
More generally, any $0<c\leq4$ can be used in place of $4$, with $c<4$ only yielding a smaller certified ellipsoidal domain, while $ c>4 $ is not guaranteed by the present saturation estimate; hence, $4$ is the largest level certified by the analysis rather than a controller tuning parameter.
}
Then, by using Items 4 and 6 of Theorem \ref{theoP}, the time derivative of $ V $ along the trajectories of system \eqref{systemclosestate} can be evaluated as follows
\begin{align*}
    \dot{V}=&
    2n\dot{\gamma}x^{\mathrm{T}}Px+2n\gamma x^{\mathrm{T}}\dot{P} x +2n\gamma x^{\mathrm{T}}(A^{\mathrm{T}}P+PA)x  \\
    & -4n\gamma \sum_{i=1}^{m}x^{\mathrm{T}}Pb_{i}\sigma\left(\left( \frac{1}{2}+\eta \right)  b_{i}^{\mathrm{T}}Px \right)   \\
    \le & 2n\dot{\gamma}x^{\mathrm{T}}Px +\dot{\gamma}\delta_{\mathrm{c}}  x^{\mathrm{T}}Px +2n\gamma x^{\mathrm{T}}(A^{\mathrm{T}}P+PA)x \\
    &- 4n \gamma  \sum_{i=1}^{m}
    x^{\mathrm{T}}Pb_{i}\sigma\left( \frac{1}{2} b_{i}^{\mathrm{T}}Px\right) \\
    =& 2n\dot{\gamma}x^{\mathrm{T}}Px +\dot{\gamma}\delta_{\mathrm{c}}  x^{\mathrm{T}}Px +2n\gamma x^{\mathrm{T}}(A^{\mathrm{T}}P+PA)x\\
    &- 2n \gamma  \sum_{i=1}^{m}
    x^{\mathrm{T}}Pb_{i}  b_{i}^{\mathrm{T}}Px\\
    = & 2n\dot{\gamma}x^{\mathrm{T}}Px+ \dot{\gamma}\delta_{\mathrm{c}} x^{\mathrm{T}}Px + 2n \gamma  x^{\mathrm{T}}(A^{\mathrm{T}}P+PA\\
    & -PBB^{\mathrm{T}}P)x\\
    \le & 2n\dot{\gamma}x^{\mathrm{T}}Px +\dot{\gamma}\delta_{\mathrm{c}} x^{\mathrm{T}}Px-  2n\gamma^{2} x^{\mathrm{T}}Px\\
    = & \left( \frac{\dot{\gamma}}{ \gamma }   +\frac{\delta_{\mathrm{c}}}{2n\gamma }\dot{\gamma} - \gamma \right) 2n\gamma  x^{\mathrm{T}}Px\\
    = & \left( \frac{2n+\delta_{\mathrm{c}}}{2n\gamma }\dot{\gamma} - \gamma\right) V.
\end{align*}
Based on \eqref{defgamat}, we have $\dot{\gamma}=\gamma^{2}/(T\gamma_{0}).$
Then there holds
\begin{equation*}
    \dot{V} \le \left( \frac{2n+\delta_{\mathrm{c}}}{2nT\gamma_{0}} -1  \right)   \gamma V \triangleq  -\varepsilon_{2} \gamma V,% \label{eqdVleV}
\end{equation*}
where $\varepsilon_{2} >0$ in view of (\ref{defT}).
Therefore, condition (\ref{eqzb27}) in Lemma \ref{lemmaefts} is satisfied with $c=4$ and $ \mu(t)=-\varepsilon_{2}\gamma(t) $.

It can be verified that
\begin{equation*}
    \rho(t)=\int_{0}^{t}\mu(s)\mathrm{d}s -\ln(\kappa_{1}(t)),
\end{equation*}
is a  $ T $-FTS function.
Therefore, the function  $ V(t,x) $  satisfies all the conditions in Lemma \ref{lemmaefts}.
Thus the closed-loop system (\ref{systemclosestate}) is exact $T$-FTS with $\mathscr{E}(\gamma_{0})$ being contained in the domain of attraction.
Moreover, it follows from Lemma \ref{lemmaefts} that the norm of the state  $ x(t) $  of the closed-loop system \eqref{systemclosestate} satisfies
{
\begin{align*}
    \| x(t) \|^{2} \le & \frac{\kappa_{2}(0)}{\kappa_{1}}\exp\left( \int_{0}^{t}\mu(s)\mathrm{d}s  \right) \| x(0) \|^{2} \\
    = & \frac{\kappa_{2}(0)(T-t)^{\varepsilon_{2}\gamma_{0}T}}{\kappa_{1}T^{\varepsilon_{2}\gamma_{0}T}} \| x(0) \|^{2},\ \forall t \in [0,T).
\end{align*}}
In addition, we have
{
\begin{align*}
    \| x(t) \|^{2}\!\ge &  \frac{\kappa_{1}}{\kappa_{2}(t)}\exp\left( \int_{0}^{t}\nu(s)\mathrm{d}s  \right)\| x(0) \|^{2}  \\
    = & \frac{\kappa_{1}(T-t)^{\varepsilon_{1}\gamma_{0}T}}{\kappa_{2}(t)T^{\varepsilon_{1}\gamma_{0}T}}  \|x(0)\|^{2}\\
    =& \frac{\kappa_{1}(T\!\!-\!t)^{2\mu_{\mathrm{c}}+\varepsilon_{1}\gamma_{0}T}}{2n\varphi(\gamma_{0})\gamma_{0}^{2\mu_{\mathrm{c}}}T^{2\mu_{\mathrm{c}}+\varepsilon_{1}\gamma_{0}T}} \|x(0)\|^{2}
    ,\ \forall t \in [0,T).
\end{align*}}
Thus, \eqref{xle} is proven.

Finally, it follows from (\ref{eqdVleV0}) that, for all $t \in [0,T)$,
{
\begin{align*}
    \| u(t) \|^{2}\le & \| U(t) \|^{2}=\left( \frac{1}{2}+\eta \right)^{2} x^{\mathrm{T}}PBB^{\mathrm{T}}Px \\
    \le& \left( \frac{1}{2}+\eta \right)^{2} 2n  \gamma  x^{\mathrm{T}}Px = \left( \frac{1}{2}+\eta \right)^{2} V(t,x(t))\\
    \le & \left( \frac{1}{2}+\eta \right)^{2} \exp\left( \int_{0}^{t}\mu(s)\mathrm{d}s  \right) V(0,x(0))\\
    \le & \frac{\left( \frac{1}{2}+\eta \right)^{2}2n\gamma_{0}\left\| P(\gamma_{0}) \right\| (T-t)^{\varepsilon_{2}\gamma_{0}T} }{T^{\varepsilon_{2}\gamma_{0}T}} \|x(0)\|^{2}.
\end{align*}}
Equation \eqref{ule} is now verified and the proof is finished.
\end{proof}

{
\begin{remark}
The bound in \eqref{defT} is a simple closed-form sufficient condition associated with the reciprocal schedule in \eqref{defgamat}.
Its conservatism can be reduced by exploiting the pointwise derivative bound in Item 6 of
Theorem~\ref{theoP}.
Define
\[
\psi_f(\gamma)=1+\frac{f(\gamma)}{2n},\quad
T_f(\gamma_0)=
\int_{\gamma_0}^{\infty}
\frac{\psi_f(s)}{s^2}\mathrm{d}s.
\]
For any prescribed time $ T>T_f(\gamma_0) $, choose $\gamma(t)$ according to
\begin{equation}
    \dot{\gamma} =
    \frac{T_f(\gamma_0)}{T}
    \frac{\gamma^2}{\psi_f(\gamma)},
    \quad
    \gamma(0)=\gamma_0.
    \label{refinedgamma}
\end{equation}
Then $ t= T/T_f(\gamma_0) \int_{\gamma_0}^{\gamma(t)} \psi_f(s)/s^2\mathrm{d}s $.
Since $\psi_f(s)>0$ and $ T_f(\gamma_0) = \int_{\gamma_0}^{\infty}\psi_f(s)/s^2 \mathrm{d}s $, the integral on the right-hand side is strictly smaller than $ T_f(\gamma_0) $ for every finite $ \gamma(t) $ and approaches $ T_f(\gamma_0) $ only as $ \gamma(t)\rightarrow\infty $.
Hence, $\gamma(t)\rightarrow\infty$ as $t\uparrow T$.
Moreover, the proof of Theorem~\ref{theostate} gives $ \dot V \leq -( 1- T_f(\gamma_0)/T )\gamma V $.
Since $\psi_f(\gamma)\geq1$, it follows from \eqref{refinedgamma} that $ \int_0^T\gamma(t)\mathrm{d}t=\infty $.
Hence, the exact prescribed-time convergence property is retained.
Furthermore, since $ f(\gamma)\leq\delta_{\mathrm c}$, $ T_f(\gamma_0) \leq (2n+\delta_{\mathrm c})/ 2n\gamma_0$.
Thus, the pointwise bound $f(\gamma)$ provides a less conservative prescribed-time design than the uniform estimate based on $\delta_{\mathrm c}$.
\label{remarkgama}
\end{remark}
}

{
\begin{remark}
Although $\gamma(t)\rightarrow\infty$ as $t\uparrow T$, the singularity at $t=T$ can be avoided in practical implementation by truncating the time-varying gain, similarly to \cite{zhou20auto1}.
Specifically, for any sufficiently small $\varepsilon>0$, Theorem~\ref{theostate} guarantees that there exists a $T_{\mathrm{c}}<T$ such that $\|x(t)\|\leq\varepsilon$ for $t\in[T_{\mathrm{c}},T)$.
Define
\begin{equation*}
    \bar{\gamma}(t)=
    \begin{cases}
        \gamma(t), & 0\leq t<T_{\mathrm{c}},\\
        \gamma(T_{\mathrm{c}}) \triangleq \gamma_{\mathrm{c}}, & t\geq T_{\mathrm{c}}.
    \end{cases}
\end{equation*}
Then the implementable feedback is $ U(t,x)=-\left(1/2+\eta\right) B^{\mathrm T} P(\bar{\gamma}(t))x,\ t\geq0 $.
For $t\geq T_{\mathrm{c}}$, $\bar{\gamma}(t)=\gamma_{\mathrm{c}}$ and $\dot{\bar{\gamma}}(t)=0$.
With $V=2n\bar{\gamma}x^{\mathrm T}P(\bar{\gamma})x$, it follows from the estimate in the proof of Theorem~\ref{theostate} that $\dot V\leq-\gamma_{\mathrm{c}}V,\ t\geq T_{\mathrm{c}}$.
Therefore, after the gain truncation, the state continues to converge exponentially to zero. In particular, the control input is not removed at $t=T$, so residual states remain stabilized after the prescribed time. In practice, $T_{\mathrm{c}}$ can simply be chosen sufficiently close to $T$.
\end{remark}
}

\section{\label{sec4}A Numerical Example}

Consider a linear system in the form of (\ref{system}) with \cite{sirisena1968optimal}
{
\begin{equation}
    \dot{x}(t)=\left[\begin{array}{cccc}
     0& 1&0 &0 \\
     0& 0& 1&0 \\
     0& 0&0 &1 \\
     0&0 &-2 &-2
    \end{array}\right]x(t)+
    \left[\begin{array}{c}
    0 \\
    0 \\
    0 \\
    1
    \end{array}\right]u.
    \label{syssi}
\end{equation}
It can be verified that $(A,B)$ is controllable and $ \lambda(A)=\{0,0,-1+\mathrm{i},-1-\mathrm{i}\}$.
Therefore, Assumption~\ref{assump1} is satisfied for this system.

The feedback command used in the simulations is designed as $U(t,x)=-(1/2+\eta)B^{\mathrm T}P(\bar{\gamma}(t))x$, where $\eta\geq0$ is a free feedback-gain parameter and $\bar{\gamma}(t)$ denotes the practically implemented gain schedule.
For the saturated control case, the actual input applied to the system is given by $u(t)=\sigma(U(t,x))$.
The ideal gain schedule $\gamma(t)$ is obtained from \eqref{refinedgamma}.
For practical implementation, the gain schedule is truncated at $T_{\mathrm c}=T-0.2$~s; specifically, $\bar{\gamma}(t)=\gamma(t)$ for $0\leq t<T_{\mathrm c}$ and $\bar{\gamma}(t)=\gamma_{\mathrm c}$ for $t\geq T_{\mathrm c}$, where $\gamma_{\mathrm c}=\gamma(T_{\mathrm c})$.
Thus, $\gamma_{\mathrm c}$ is determined by the gain schedule rather than independently selected.

The parameter $\eta$ does not affect the validity of Theorem~\ref{theostate}.
Increasing $\eta$ increases the magnitude of the unsaturated control command and may affect the transient behavior, while it may also increase the actuator demand and the possibility of actuator saturation.
In the simulations, $\eta=5$ is selected as an illustrative value rather than an optimized one.

The matrix $P(\gamma)$ in the above controller is evaluated using its explicit polynomial representation.
For this example, the parameterized linear matrix equations \eqref{lyap0} and \eqref{lyap00} are vectorized by means of Kronecker products and solved analytically as functions of $\gamma$ in the design stage.
After algebraic simplification, the resulting positive definite matrix $P(\gamma)=Y(\gamma)-X(\gamma)$ admits the following explicit polynomial representation:
% The feedback command is designed as $U(t,x)=-(1/2+\eta)B^{\mathrm T}P(\bar{\gamma}(t))x$, where $\bar{\gamma}(t)$ is obtained by truncating the refined schedule \eqref{refinedgamma} at a sufficiently large finite value $\gamma_{\mathrm{c}}$ as described in the above remark.
% For the saturated control case, the actual input applied to the system is given by $u(t)=\sigma(U(t,x))$.

% The parameter $\eta\ge0$ in the feedback law is a free feedback-gain parameter and does not affect the validity of Theorem~\ref{theostate}.
% Increasing $\eta$ increases the magnitude of the unsaturated control command and may affect the transient behavior, but it also tends to increase the demand on the actuator and the possibility of actuator saturation.
% In the simulations, $\eta=5$ is selected as an illustrative value rather than an optimized one.

% In the numerical implementation, the explicit polynomial expression of $P(\gamma)$ given above is evaluated directly at the prescribed value of $\gamma$.
% Hence, the parameterized matrix equations are not numerically resolved for each $\gamma$, and $P(\gamma)$ is not formed by subtracting separately computed matrices $X(\gamma)$ and $Y(\gamma)$.
% For this example, the parameterized linear matrix equations \eqref{lyap0} and \eqref{lyap00} are vectorized by means of Kronecker products and solved analytically with respect to $\gamma$ in the design stage.
% After algebraic simplification, the resulting positive definite matrix $P(\gamma)=Y(\gamma)-X(\gamma)$ admits the following explicit polynomial representation:
}
\begin{equation*}
    P=\left[\begin{array}{cccc}
    2\gamma^{3}p_{11} & 4\gamma^{2}p_{12} & 2\gamma^{2}p_{13} & 8\gamma^{2}%
    p_{14}\\
    4\gamma^{2}p_{12} & 4\gamma p_{22} & 4\gamma p_{23} & 8\gamma p_{24}\\
    2\gamma^{2}p_{13} & 4\gamma p_{23} & 4\gamma p_{33}   & 12\gamma \left(  2\gamma+1\right)  \\
    8\gamma^{2}p_{14} & 8\gamma p_{24} & 12\gamma \left(  2\gamma+1\right)   &
    8\gamma
    \end{array}
    \right],
\end{equation*}
where $p_{ij}=p_{ij}(\gamma)$ are some polynomials defined by
\begin{align*}
p_{11}  & =127\gamma^{4}+252\gamma^{3}+248\gamma^{2}+120\gamma+28,\\
p_{12}  & =66\gamma^{4}+103\gamma^{3}+76\gamma^{2}+26\gamma+4,\\
p_{13}  & =57\gamma^{3}+74\gamma^{2}+46\gamma+8,\\
p_{14}  & =2\gamma^{2}+2\gamma+1,\\
p_{22}  & =119\gamma^{4}+164\gamma^{3}+102\gamma^{2}+32\gamma+4,\\
p_{23}  & =50\gamma^{3}+51\gamma^{2}+22\gamma+4,\\
p_{24}  & =4\gamma^{2}+3\gamma+1,\\
p_{33}  & =37\gamma^{2}+30\gamma+8.
\end{align*}
{
Thus, $P(\gamma)$ is a polynomial matrix and can be evaluated directly during the simulation without repeatedly solving the parameterized matrix equations.
For the refined gain design, define $\delta(\gamma)$ as the maximum eigenvalue of $2n\gamma(\mathrm{d}P/\mathrm{d}\gamma)P^{-1}$.
The functions $\delta(\gamma)$ and $f(\gamma)$, together with the uniform bound $\delta_{\mathrm c}$, are shown in Fig.~\ref{figdetac}.
The numerical evaluation of $\sup_{\gamma>0} \{\lambda_{\max}( 2n\gamma \frac{\mathrm{d}P}{\mathrm{d}\gamma}P^{-1} )  \} $ gives $\delta_{\mathrm c}=60.91$.
As guaranteed by Item~6 of Theorem~\ref{theoP}, $\delta(\gamma)\leq f(\gamma)\leq\delta_{\mathrm c}$.
}
% It follows that $P$ is a polynomial matrix of $\gamma$.

\begin{figure}[htbp]
\centering
\includegraphics[scale=1]{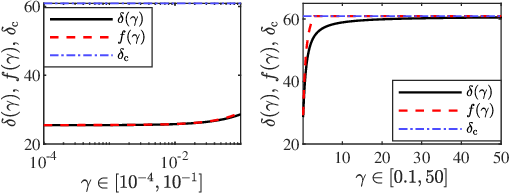}
\caption{The pointwise factors $\delta(\gamma)$ and $f(\gamma)$ and the uniform bound $\delta_{\mathrm c}$ for system \eqref{syssi}.}
\label{figdetac}
\vspace{-1em}
\end{figure}

{
Let the initial set be given by $\mathit{\Omega}(D)=\{x\in\mathbf{R}^{4}\mid\|x\|\leq D\}$, where $D>0$.
For each radius $D$, the initial condition is selected as $x(0)=\beta[1,-1,1,-1]^{\mathrm T}$.
Since $\|[1,-1,1,-1]^{\mathrm T}\|=2$, choosing $\beta=D/2$ places $x(0)$ on the boundary of $\mathit{\Omega}(D)$.
For each prescribed initial set $\mathit{\Omega}(D)$, the parameter $\gamma_0>0$ is selected from the admissible values satisfying $\mathit{\Omega}(D)\subseteq\mathscr{E}(\gamma_0)$.
According to \eqref{defEs}, this condition is equivalent to $2n\gamma_0\lambda_{\max}(P(\gamma_0))D^2\leq4$.
Among the admissible values, a relatively large $\gamma_0$ is preferred because $T_f(\gamma_0)$ in Remark~\ref{remarkgama} is strictly decreasing with respect to $\gamma_0$.
Therefore, a larger admissible $\gamma_0$ results in a smaller sufficient prescribed-time bound.
The prescribed time $T$ is then selected as the smallest integer strictly larger than $T_f(\gamma_0)$.
}
\begin{center}
\begin{table}[htbp]
    \centering
  \textbf{Table 1}: {Simulation settings for different values of \(D\)}\\
\begin{tabular}{c|c|c|c|c} \hline \hline
$D$&1&2&3&4\\\hline
$\beta$&0.5&1&1.5&2\\
$\gamma_{0}$&0.0856&0.0466&0.0320&0.0244\\
${T_{f}(\gamma_{0})}$&{63.74}&{107.38}&{149.97}&{191.85}\\
$T$&{64}&{108}&{150}&{192}\\
${T_{\mathrm c}}$&{63.8}&{107.8}&{149.8}&{191.8}\\
${\gamma_{\mathrm c}}$&{43.2475}&{43.3179}&{43.0777}&{43.1034}\\
\hline \hline
\end{tabular}
\vspace{-2em}
\end{table}
\end{center}
{
The corresponding state and control responses for different prescribed initial sets $\mathit{\Omega}(D)$ are shown in Fig.~\ref{figx}.
For all considered values of $D$, the state trajectories converge to a neighborhood of the origin within the prescribed time $T$, and remain bounded thereafter under the frozen gain, while the actual control input satisfies the prescribed saturation constraint.
Moreover, as $D$ increases, the prescribed initial set becomes larger, while the corresponding sufficient time bound $T_f(\gamma_0)$, and hence the selected prescribed time $T$, increase.
% Once the gain reaches $\gamma_{\mathrm{c}}$, it is kept constant thereafter, so that the controller remains well defined for $t\geq T$.
% The corresponding state and control responses for different prescribed initial sets $\mathit{\Omega}(D)$ are shown in Fig.~\ref{figx}.
% The numerical results show that, for all considered values of \(D\), the state trajectories exhibit rapid convergence to the origin near the prescribed time \(T\), highlighting the prescribed-time convergence behavior predicted by the theoretical analysis.
% Moreover, as $ D $ increases, the prescribed initial set becomes larger, while the corresponding sufficient time bound $T_f(\gamma_0)$, and hence the selected prescribed time $ T $, increase.
}

\begin{figure}[h]
\begin{center}
   \includegraphics[scale=1]{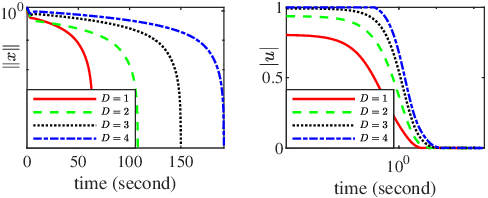}
\end{center}
\caption{State norm and control input magnitude under the proposed constrained controller for different values of \(D\).}
\label{figx}
\end{figure}

{
To further evaluate the performance of the proposed time-varying high-gain controller, a comparison with the prescribed-time control method in \cite{song2017time} is provided.
The method in \cite{song2017time} is a representative time-varying high-gain feedback approach, where the prescribed-time convergence is achieved by introducing a time-varying scaling function.
For completeness and reproducibility, the compared controller in \cite{song2017time} is implemented according to its original design procedure, and the detailed controller parameters are provided in the Appendix~C.
For a task-level comparison, the method in \cite{song2017time} is implemented according to its original unconstrained formulation, using the same initial conditions and the same prescribed terminal times \(T\) as those listed in Table 1.
The purpose of this comparison is to examine whether the original controller designs can accomplish the same prescribed-time regulation task without violating the prescribed actuator-amplitude constraint.
We do not add an external saturation map to the controller in \cite{song2017time}, because doing so would alter its closed-loop dynamics and its original prescribed-time convergence guarantee would no longer automatically apply.
Accordingly, the comparison focuses on the actuator demand generated by the original unconstrained controller in \cite{song2017time} and the constraint-handling capability of the proposed controller.

To compare the constrained input of the proposed controller with the actuator demand of \cite{song2017time}, let \(U_2(t)\) denote the unconstrained command generated by \cite{song2017time}, while \(U(t,x)\) and \(u(t)=\sigma(U(t,x))\) denote respectively the command and actual input of the proposed controller.
Define \(u_{\max}=\max_{0<t<T}|u(t)|\), \(U_{2,\max}=\max_{0<t<T}|U_2(t)|\), \(E_u=\int_0^T u^2(t)\,dt\), and \(E_{U_2}=\int_0^T U_2^2(t)\,dt\).
These quantities are reported in Table 2, where \(U_{2,\max}>1\) indicates that the command required by \cite{song2017time} exceeds the prescribed actuator bound.
The corresponding responses of \cite{song2017time} are shown in Fig.~\ref{figsong}.
}

\begin{center}
\begin{table}[htbp]
    \centering
  \textbf{Table 2}: {Comparison between the actual saturated input \(u\) of the proposed controller and the unconstrained input demand \(U_2\) of \cite{song2017time}.}\\
    \begin{tabular}{c|c|c|c|c}
\hline\hline
$D$ & 1 & 2 & 3 & 4\\
\hline
{$u_{\max}$} & {0.8037 }& {0.9373 }& {0.9923 }& {1.0000}\\
{$U_{2,\max}$} & {0.4595 }& {0.8786} &{1.2958 }&{1.7120}\\
{$E_u$ }& {0.1723 }& {0.4785 }& {0.8121 }& {1.1468}\\
{$E_{U_{2}}$ }& {0.1437 }& {0.5452 }& {1.2048} & {2.1221}\\
\hline\hline
\end{tabular}
\end{table}
\end{center}
\vspace{-3em}
\begin{figure}[htbp]
\begin{center}
   \includegraphics[scale=1]{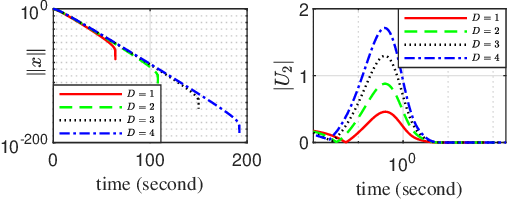}
\end{center}
\caption{State norm and unconstrained control-command magnitude \(|U_2|\) under the time-varying high-gain method in \cite{song2017time} for different values of \(D\).}
\label{figsong}
\end{figure}

{
Both methods are simulated using the same initial conditions and prescribed terminal times.
The common value of \(T\) defines an identical regulation task and is not used to claim that either method has a shorter actual settling time.
For \(D=1\), \(U_{2,\max}<u_{\max}\), whereas the two peak values are comparable for \(D=2\).
For \(D=3\) and \(D=4\), however, the peak values of the unconstrained command \(U_2\) required by \cite{song2017time} are \(1.2958\) and \(1.7120\), respectively.
These values exceed the prescribed actuator bound and show that the original controller in \cite{song2017time} cannot be directly implemented by the considered actuator without modifying its control law.
In contrast, the actual input \(u=\sigma(U)\) of the proposed controller remains within \([-1,1]\), while the state is driven to a neighborhood of the origin within the prescribed horizon.
Moreover, for these two cases, the unconstrained command energies \(E_{U_2}\) required by \cite{song2017time} are higher than the actual input energies \(E_u\) of the proposed controller.
These results demonstrate the advantage relevant to the bounded-input problem considered in this paper.
The proposed controller explicitly incorporates the actuator saturation into its design and accomplishes the prescribed regulation task with an admissible physical input, whereas the original unconstrained controller in \cite{song2017time} may demand an input that exceeds the available actuator amplitude.
}

\section{Conclusion}

This paper investigated the problem of semi-global exact prescribed-time stabilization for linear systems with bounded controls.
{
A novel nested PLE-based approach was first established to construct a parameterized control Lyapunov function whose level set can be made arbitrarily large and the real parts of the poles of the closed-loop system approach minus infinity as the parameter approaches infinity.
By carefully investigating the properties of the nested PLEs, a linear time-varying high-gain feedback controller was designed to guarantee semi-global exact prescribed-time stabilization for any prescribed bounded set of initial conditions and all sufficiently large prescribed settling times.
A numerical example and a comparative simulation with an existing prescribed-time control method were provided to demonstrate the effectiveness of the proposed approach and illustrate its performance under bounded control constraints.

}

\section*{Appendix}

\subsection*{A: The Proof of Lemma \ref{lm1}}

We consider the optimal control problem associated with system (\ref{sys2}) and the following objective function $ J=\int_{0}^{t_{f}}\mathrm{d}t=t_{f} $.
According to the Pontryagin's Minimum Principle, we construct the Hamiltonian function {$ H( x,u,\lambda ) = 1+\lambda(-\alpha_{0}x +u) $}, where $\lambda$ is the costate variable.
The optimal control that minimizes $J$ subject to the terminal constraint $x(t_f)=0$ satisfies
\begin{equation}
    u^{*}(t)=\mathrm{arg}\min_{\vert u \vert\le 1 } H( x^{*},u,\lambda^{*}).
    \label{uop}
\end{equation}
The costate equation is $\dot{\lambda}=-\frac{\partial H}{\partial x}=\alpha_{0} \lambda$, which has the solution $\lambda(t)=\lambda_{0}\mathrm{e}^{\alpha_{0}t}$.
Based on \eqref{uop}, we can get $u^{*}(t)=-\mathrm{sgn}(\lambda)$.
Since $\lambda_{0}=0$ would lead to a trivial solution that cannot satisfy the terminal constraint $x(t_{f})=0$, it is excluded.
If $ x_{0}>0 $, based on \eqref{sys2}, in order to minimize $ t_{f} $, we should minimize $ \dot{x} $, which requires the optimal control $ u^{*}=-1 $ and $ \lambda_{0}>0 $.
If $ x_{0}<0 $, based on \eqref{sys2}, in order to minimize $ t_{f} $, we should maximize $ \dot{x} $, which requires the optimal control $ u^{*}=1 $ and $ \lambda_{0}<0 $.
By taking $ x_{0}<0 $ as an example (the case $x_{0} > 0$ can be considered in the same way), the closed-loop system is $\dot{x}(t)=-\alpha_{0}x(t)+1$ from which we get $x(t_f)=0$ with $ t_{f}=\frac{1}{\alpha_{0}}\ln(1-\alpha_{0}x_{0})$, implying (\ref{eqzb2}) by noting that $x_0 \leq D$.
The proof is finished.
\subsection*{B: The Proof of Theorem \ref{theoP}}
\textbf{\emph{Proof of Item 1}}:
The PLE \eqref{ple0} can be written as (\ref{lyap0}), where $ \tilde{X}=X^{-1} $.
By the Lyapunov equation theory, equation \eqref{lyap0} has a unique positive definite solution if and only if $\alpha( A-(\gamma/2) I_{n}) <0,$ which is always satisfied since $ \gamma>0$.
According to \eqref{ple0}, we have
\begin{equation}
    A+BB^{\mathrm{T}}X=\gamma I_{n}-X^{-1}A^{\mathrm{T}}X,
    \label{i-xax}
\end{equation}
and thus
\begin{equation}
    \lambda_{i}( A+BB^{\mathrm{T}}X ) =\gamma - \lambda_{i}( A ),\; i=1,2,\ldots,n.
    \label{eqeigABBX}
\end{equation}
Similarly,  by denoting $ \tilde{Y}=Y^{-1} $, the PLE \eqref{ple1} can be written as (\ref{lyap00}), which has a unique positive definite solution if and only if $ \alpha( -A-BB^{\mathrm{T}}X- (\gamma /2)I_{n} )=\alpha( A )- \gamma - \gamma/2 <0 $ from \eqref{eqeigABBX}, which always holds due to $ \gamma>0 $.
Since both $X$ and $Y$ are rational matrices from Lemma \ref{lm2}, so is $P$.

According to the PLE \eqref{ple1}, we have
\begin{equation}
A-BB^{\mathrm{T}}P=- \gamma I_{n}-Y^{-1}(A+BB^{\mathrm{T}}X)^{\mathrm{T}}Y,
\label{eqpleABBP}
\end{equation}
from which and (\ref{eqeigABBX}) it follows that
\begin{align*}
    \lambda_{i}(A-BB^{\mathrm{T}}P)=&
    \lambda_{i}(- \gamma I_{n}-Y^{-1}(A+BB^{\mathrm{T}}X)^{\mathrm{T}}Y ) \notag\\
    =& -\gamma -\lambda_{i}(A+BB^{\mathrm{T}}X)\notag\\
    =& \lambda_{i}(A)-2\gamma,\; i=1,2,\ldots,n,
  \end{align*}
which proves (\ref{poles}).

\textbf{\emph{Proof of Item 2}}:
Combining \eqref{ple0} and \eqref{ple1} gives
\begin{equation}
    A^{\mathrm{T}}P+PA-PBB^{\mathrm{T}}P = - \gamma  (X+ Y).
    \label{ARE}
\end{equation}
Since $(A,B)$ is controllable, $P$ is a stabilizing solution (according to (\ref{poles})), and $X+Y>0$, it follows from the well-known Riccati theory (see, for example, \cite{willems03tac}) that $P=Y-X $ is the unique positive definite stabilizing solution to the above Riccati equation (\ref{ARE}).

\textbf{\emph{Proof of Item 3}}: By using (\ref{i-xax}), \eqref{eqeigABBX} and (\ref{eqpleABBP}), we have
\begin{align*}
    \mathrm{tr}( B^{\mathrm{T}}PB ) =& \mathrm{tr}( BB^{\mathrm{T}}P )\\
    =&\mathrm{tr}( A )+n\gamma+\mathrm{tr}( Y^{-1}(A+BB^{\mathrm{T}}X)^{\mathrm{T}}Y )\\
    =&\mathrm{tr}( A )+n\gamma+\mathrm{tr}( A+BB^{\mathrm{T}}X ) \\
    =& \mathrm{tr}( A )+n\gamma-\mathrm{tr}( A )+n\gamma\\
    =& 2n\gamma.
\end{align*}

\textbf{\emph{Proof of Item 4}}:
According to \eqref{ARE}, we have
\begin{align*}
A^{\mathrm{T}}P+PA-PBB^{\mathrm{T}}P  \le -   \gamma  (Y-X) =-  \gamma  P,
\end{align*}
which is the inequality on the right hand side of (\ref{pli}).
Thus it remains to prove the inequality on the left hand side of (\ref{pli}), which is equivalent to $ \lambda_{\max}( ( Y-X )^{-\frac{1}{2}} ( Y+X ) ( Y-X )^{-\frac{1}{2}}  \gamma/(\gamma-\phi(A)))\le k$ for some $k>0$.
Thus we need to prove
\begin{align}
 \lim_{\gamma \to \infty}\lambda_{\max}\left( ( Y+X ) ( Y-X ) ^{-1} \right)  &< \infty, \label{kbig}\\
    \lim_{\gamma \downarrow 0} \lambda_{\max}\left(\frac{\gamma ( Y+X ) ( Y-X ) ^{-1}}{\gamma-\phi(A)}  \right) &< \infty.
    \label{ksmal}
\end{align}
{
We first establish the following lemma.
Since the eigenvalues involved in \eqref{kbig} are invariant under nonsingular state transformations, without loss of generality, we transform $(A,B)$ into its Luenberger canonical form when proving \eqref{kbig}.
}
\begin{lemma}
    Suppose that  $ (A,B) $  is in its Luenberger canonical form.
    Define
\begin{equation*}
    L(\gamma)=\bigoplus_{i=1} ^{m} L_{\mu_{i}}(\gamma), \ L_{k}(\gamma)=\bigoplus_{l=1}^{k}\gamma^{k-l},
\end{equation*}
where $\{\mu_1,\ldots,\mu_{m}\}$ is the controllability index set.
Then
\begin{gather}
    \lim_{\gamma \rightarrow\infty}\gamma^{-1}L^{-1}XL^{-1}=\varPi_{X}^{\infty},
    \label{eqLx}\\
    \lim_{\gamma \rightarrow\infty}\gamma^{-1}L^{-1}YL^{-1}=\varPi_{Y}^{\infty},
    \label{eqLy}\\
    \lim_{\gamma \rightarrow\infty}\gamma^{-1}L^{-1}PL^{-1}=\varPi_{\infty},
    \label{eqLp}
\end{gather}
where $ \varPi_{X}^{\infty}$, $\varPi_{Y}^{\infty} $ and $ \varPi_{\infty} $ are positive definite matrices.
\end{lemma}
{
\begin{proof}
By applying a nonsingular input transformation and removing redundant input channels, we may assume without loss of generality that \(B\) has full column rank.
Since $(A,B)$ is in its Luenberger canonical form, we can write
\[
B=\left(\bigoplus_{i=1}^{m}e_{\mu_i}\right)\varGamma,
\]
where $\varGamma\in\mathbf{R}^{m\times m}$ is nonsingular and $ e_{\mu_{i}}=[0,\cdots,0,1]^\mathrm{T} \in \mathbf{R}^{\mu_i} $.
By the definition of $L(\gamma)$, we have $LB=B$.
Define
\begin{equation}
    \mathcal{A}=\gamma^{-1}LAL^{-1},
    \quad
    \varPi_X=\gamma^{-1}L^{-1}XL^{-1}.
    \label{defPiX}
\end{equation}
Premultiplying and postmultiplying \eqref{ple0} by
$\gamma^{-1}L^{-1}$ gives
\begin{equation}
    \mathcal{A}^{\mathrm T}\varPi_X
    +\varPi_X\mathcal{A}
    +\varPi_XBB^{\mathrm T}\varPi_X
    =\varPi_X.
    \label{eqPiX}
\end{equation}
By the structure of the Luenberger canonical form,
\[
    \lim_{\gamma\to\infty}\mathcal{A}
    =\varPsi
    \triangleq
    \bigoplus_{i=1}^{m}\varPsi_i,
\]
where
\[
    \varPsi_i=
    \begin{bmatrix}
    0&1&&\\
    \vdots&&\ddots&\\
    0&&&1\\
    0&0&\cdots&0
    \end{bmatrix}
    \in\mathbf{R}^{\mu_i\times\mu_i}.
\]
Moreover, $(\varPsi,B)$ is controllable.
Let $W_X=\varPi_X^{-1}$. Multiplying \eqref{eqPiX} from the left and right by $W_X$ yields
\begin{equation}
    \left(\mathcal{A}-\frac{1}{2}I_n\right)W_X
    +W_X\left(\mathcal{A}-\frac{1}{2}I_n\right)^{\mathrm T}
    =-BB^{\mathrm T}.
    \label{eqWX}
\end{equation}
Since $\mathcal{A}$ is similar to $A/\gamma$, $ \lambda_i(\mathcal{A})=\lambda_i(A)/\gamma $.
By Assumption~\ref{ass1}, $\Re\{\lambda_i(A)\}\le0$; hence $\mathcal{A}-\frac{1}{2}I_n$ is Hurwitz for every $\gamma>0$.
Therefore, \eqref{eqWX} has the unique positive definite solution
\[
    W_X=
    \int_0^\infty
    e^{(\mathcal{A}-\frac{1}{2}I_n)t}
    BB^{\mathrm T}
    e^{(\mathcal{A}-\frac{1}{2}I_n)^{\mathrm T}t}\,{\rm d}t.
\]
Since $\mathcal{A}-\frac{1}{2}I_n \to\varPsi-\frac{1}{2}I_n$ as $\gamma\to\infty$, continuity of the solution of the Lyapunov equation gives $ W_X\to W_X^\infty $, where $W_X^\infty$ is the unique solution of
\[
    \left(\varPsi-\frac{1}{2}I_n\right)W_X^\infty
    +W_X^\infty
    \left(\varPsi-\frac{1}{2}I_n\right)^{\mathrm T}
    =-BB^{\mathrm T}.
\]
Because $(\varPsi,B)$ is controllable, $W_X^\infty>0$.
Consequently, $ \varPi_X \to \varPi_X^\infty\triangleq(W_X^\infty)^{-1}>0 $, which proves \eqref{eqLx}.

Next, define $ A_X=\mathcal{A}+BB^{\mathrm T}\varPi_X, \ \varPi_Y=\gamma^{-1}L^{-1}YL^{-1} $.
It follows from \eqref{eqPiX} that $ A_X=I_n-\varPi_X^{-1}\mathcal{A}^{\mathrm T}\varPi_X $, and hence
\begin{equation}
    \lambda_i(A_X)
    =1-\frac{\lambda_i(A)}{\gamma},
    \qquad i=1,\ldots,n.
    \label{eigAX}
\end{equation}
Similarly, transforming \eqref{ple1} gives
\begin{equation}
    A_X^{\mathrm T}\varPi_Y
    +\varPi_YA_X
    -\varPi_YBB^{\mathrm T}\varPi_Y
    =-\varPi_Y.
    \label{eqPiY}
\end{equation}
Let $W_Y=\varPi_Y^{-1}$.
Then \eqref{eqPiY} is equivalent to
\begin{equation}
    \left(-A_X-\frac{1}{2}I_n\right)W_Y
    +W_Y\left(-A_X-\frac{1}{2}I_n\right)^{\mathrm T}
    =-BB^{\mathrm T}.
    \label{eqWY}
\end{equation}
By \eqref{eigAX} and Assumption~\ref{ass1}, $-A_X-\frac{1}{2}I_n$ is Hurwitz.
Moreover, $ A_X\to A_X^\infty \triangleq \varPsi+BB^{\mathrm T}\varPi_X^\infty $, and $ \lambda_i(A_X^\infty)=1,\ i=1,\ldots,n $.
Since $A_X^\infty$ is obtained from $\varPsi$ by state feedback, $(A_X^\infty,B)$ remains controllable.
Therefore, by the same Lyapunov-equation argument as above, $ W_Y\to W_Y^\infty>0 $, and consequently $ \varPi_Y\to \varPi_Y^\infty\triangleq(W_Y^\infty)^{-1}>0 $.
This proves \eqref{eqLy}.

It remains to show that the limit associated with $P=Y-X$ is positive definite.
Define $\varPi_\infty=\varPi_Y^\infty-\varPi_X^\infty $.
Passing to the limit in \eqref{eqPiX} and \eqref{eqPiY}, and subtracting the resulting equations, gives
\begin{equation}
    \varPsi^{\mathrm T}\varPi_\infty
    +\varPi_\infty\varPsi
    -\varPi_\infty BB^{\mathrm T}\varPi_\infty
    =-\varPi_Y^\infty-\varPi_X^\infty.
    \label{eqPiInf}
\end{equation}
Furthermore, from the limiting form of \eqref{eqPiY}, $ A_X^\infty-BB^{\mathrm T}\varPi_Y^\infty = -I_n-(\varPi_Y^\infty)^{-1} (A_X^\infty)^{\mathrm T}\varPi_Y^\infty $, which implies $ \lambda_i \left( A_X^\infty-BB^{\mathrm T}\varPi_Y^\infty \right) =-2,\ i=1,\ldots,n $.
Since $ A_X^\infty-BB^{\mathrm T}\varPi_Y^\infty = \varPsi-BB^{\mathrm T}\varPi_\infty $,  $\varPi_\infty$ is a stabilizing solution of \eqref{eqPiInf}.
Because $(\varPsi,B)$ is controllable and $\varPi_Y^\infty+\varPi_X^\infty>0$, the standard algebraic Riccati equation theory implies that $\varPi_\infty$ is the unique positive definite stabilizing solution of \eqref{eqPiInf}.
Hence $ \varPi_\infty>0 $.
Finally, since $ \gamma^{-1}L^{-1}PL^{-1} =\varPi_Y-\varPi_X $, we obtain $ \lim_{\gamma\to\infty} \gamma^{-1}L^{-1}PL^{-1} =\varPi_\infty>0 $, which proves \eqref{eqLp}.
\end{proof}}

Based on \eqref{eqLx}, \eqref{eqLy} and \eqref{eqLp}, we have $ \lim_{\gamma \to \infty}\lambda_{\max} $ $( ( Y+X )( Y-X )^{-1} )   =\lambda_{\max}( ( \varPi_{Y}^{\infty}+\varPi_{X}^{\infty} )( \varPi_{Y}^{\infty}-\varPi_{X}^{\infty} ) ^{-1})$, which indicates that \eqref{kbig} is proven.

{
We next establish the following lemma, which characterizes the small-$\gamma$ asymptotic behavior of $X$, $Y$, and $P$.
Similarly, since the eigenvalues involved in \eqref{ksmal} are invariant under nonsingular state transformations, we may work in the possibly complex Jordan coordinates of $A$ without loss of generality.
The corresponding transformed matrices are Hermitian, so $(\cdot)^{\mathrm H}$ is used below.
}
\begin{lemma}
Suppose that $A=J_{0} \oplus J_{1} \oplus \cdots \oplus J_{s}$, where $J_0 \in \mathbf{C}^{p_{0} \times p_{0}} $ is Hurwitz, {and, for $i=1,\ldots,s$, $J_i$ is the block Jordan matrix associated with a distinct imaginary-axis eigenvalue $\lambda_i$, namely, $\lambda(J_i)=\{\lambda_i\}$ and $\lambda_i\neq\lambda_j$ for $i\neq j$.
}
Denote
\[
    \mathcal{L}=\bigoplus_{i=0}^{s}\mathcal{L}_{i}, \;  \mathcal{L}_{0}=\gamma^{-\frac{1}{2}}I_{p_{0}},\; \mathcal{L}_{i}=\bigoplus_{j=1}^{q_{i}} L_{p_{ij}}(\gamma) ,
\]
and $\mathscr{L}=\oplus_{i=0}^{s}\mathscr{L}_{i},\mathscr{L}_{0}=I_{p_{0}},\mathscr{L}_{i}=\mathcal{L}_{i}$, where $\{p_{i1},\cdots,$ $p_{i,q_{i}}\}$ are the sizes of the Jordan matrices in $J_{i}$.
Then
\begin{gather}
    \lim_{\gamma \downarrow 0}\gamma^{-1}\mathcal{L}^{-1}X\mathcal{L}^{-1}=\bigoplus_{i=0}^{s} \varXi_{ii}\triangleq \varXi,
    \label{eqmLx}\\
    \lim_{\gamma \downarrow 0}\gamma^{-1}\mathcal{L}^{-1}Y\mathcal{L}^{-1}=\bigoplus_{i=0}^{s}\varTheta_{ii}\triangleq \varTheta,
    \label{eqmLy}\\
    \lim_{\gamma \downarrow 0}\gamma^{-1}\mathscr{L}^{-1}P\mathscr{L}^{-1}=\bigoplus_{i=0}^{s} \varPhi_{ii}\triangleq \varPhi,
    \label{eqmLp}
\end{gather}
where $ \varXi_{ii} $, $ \varTheta_{ii} $ and $ \varPhi_{ii} $ are positive definite constant matrices satisfying $ \varTheta_{00}=\varXi_{00} $ and $ \varTheta_{ii}>\varXi_{ii}, i=1,2,\ldots,s $.
\end{lemma}
{
\begin{proof}
Since a complex Jordan form is used below, $(\cdot)^{\mathrm H}$  denotes the conjugate transpose in this proof.
Partition
$B=[B_0^{\mathrm T}\ B_1^{\mathrm T}\ \cdots\ B_s^{\mathrm T}]^{\mathrm T}$ conformably with $A=J_0\oplus J_1\oplus\cdots\oplus J_s$.
For $i=1,\ldots,s$, write $J_i=J_{i1}\oplus\cdots\oplus J_{i,q_i}$ and $J_i=\lambda_i I_{p_i}+N_i$, where $N_i$ is block diagonal and each diagonal block is the nilpotent part of the corresponding Jordan block.
Denote $\mathcal B_i=\lim_{\gamma\downarrow0}\mathcal L_iB_i$.
When $p_0=0$, all statements involving the $0$-th block below are omitted.
The scaling argument follows the small-parameter analysis of PLEs in \cite{li2025derivative}, while the details needed here are given explicitly.

We first prove \eqref{eqmLx}.
Define $\mathscr X=\gamma^{-1}\mathcal L^{-1}X\mathcal L^{-1}=[\mathscr X_{ij}]$.
Since $\tilde X=X^{-1}$, equation \eqref{lyap0}, expressed in the Jordan coordinates, becomes $-J\tilde X-\tilde XJ^{\mathrm H}+\gamma\tilde X=BB^{\mathrm H}$.
Let $\tilde J=\mathcal LJ\mathcal L^{-1}$.
Since $\mathscr X^{-1}=\gamma\mathcal L\tilde X\mathcal L$, the above equation gives $ -\tilde J\mathscr X^{-1} -\mathscr X^{-1}\tilde J^{\mathrm H} +\gamma\mathscr X^{-1} =\gamma\mathcal LBB^{\mathrm H}\mathcal L $.
By the definition of $\mathcal L$, we have $\tilde J_0=J_0$ and $\tilde J_i=\lambda_iI_{p_i}+\gamma N_i$, $i=1,\ldots,s$.

Expanding the above equation blockwise and letting $\gamma\downarrow0$ shows that all off-diagonal blocks of $\mathscr X^{-1}$ converge to zero.
For $i\ne j$ with $i,j\ne0$, this follows from $\lambda_i\ne\lambda_j$.
For a block involving $J_0$, it follows from the spectral separation between the Hurwitz matrix $J_0$ and the imaginary-axis eigenvalues $\lambda_i$.

For the $(0,0)$ block, the limiting equation is $ -J_0\varXi_{00}^{-1} -\varXi_{00}^{-1}J_0^{\mathrm H} =B_0B_0^{\mathrm H} $.
For $i=1,\ldots,s$, after dividing the corresponding diagonal block equation by $\gamma$ and letting $\gamma\downarrow0$, we obtain $ -N_i\varXi_{ii}^{-1} -\varXi_{ii}^{-1}N_i^{\mathrm H} +\varXi_{ii}^{-1} =\mathcal B_i\mathcal B_i^{\mathrm H} $.
By controllability of $(A,B)$ and the PBH test, the corresponding limiting block pairs are controllable.
Since $J_0$ and $N_i-\frac{1}{2}I_{p_i}$ are Hurwitz, the Lyapunov equation theory gives $\varXi_{ii}^{-1}>0$, $i=0,1,\ldots,s$.
Consequently, $\mathscr X^{-1}\to\bigoplus_{i=0}^{s}\varXi_{ii}^{-1}$ and hence $\mathscr X\to\varXi=\bigoplus_{i=0}^{s}\varXi_{ii}>0$.
This proves \eqref{eqmLx}.

We next prove \eqref{eqmLy}. Let $\tilde Y=Y^{-1}$ and define $Z=XY^{-1}X$ and $D=X-Z$. Since $P=Y-X>0$ by Item~2, we have $Y>X>0$ and hence $Y^{-1}<X^{-1}$.
Therefore, $0<Z<X$ and $D>0$.

By \eqref{i-xax}, $J+BB^{\mathrm H}X=\gamma I_n-X^{-1}J^{\mathrm H}X$.
Substituting this relation into \eqref{lyap00}, and then premultiplying and postmultiplying the resulting equation by $X$, gives $ J^{\mathrm H}Z+ZJ-3\gamma Z=-XBB^{\mathrm H}X $.
On the other hand, \eqref{ple0} can be written as $J^{\mathrm H}X+XJ-\gamma X=-XBB^{\mathrm H}X$.
Subtracting the preceding two equations and using $D=X-Z$ gives
\[
\left(J-\frac{3\gamma}{2}I_n\right)^{\mathrm H}D
+D\left(J-\frac{3\gamma}{2}I_n\right)
=-2\gamma X.
\]

Define $\mathscr D=\gamma^{-1}\mathcal L^{-1}D\mathcal L^{-1}$.
Since $0<D<X$, we have $0<\mathscr D<\mathscr X$; hence $\mathscr D$ is bounded as $\gamma\downarrow0$.
With $\tilde J=\mathcal LJ\mathcal L^{-1}$, the scaled form of the above Lyapunov equation is $ \tilde J^{\mathrm H}\mathscr D+\mathscr D\tilde J -3\gamma\mathscr D=-2\gamma\mathscr X $.

For the $(0,0)$ block, letting $\gamma\downarrow0$ gives $J_0^{\mathrm H}\mathscr D_{00}+\mathscr D_{00}J_0\to0$.
Since $J_0$ is Hurwitz, $\mathscr D_{00}\to0$.
For $i=0$, $j\ne0$, spectral separation gives $\mathscr D_{0j}\to0$, and similarly $\mathscr D_{i0}\to0$.
For $i\ne j$ with $i,j\ne0$, the limiting coefficient is $\bar\lambda_i+\lambda_j=\lambda_j-\lambda_i\ne0$, and hence $\mathscr D_{ij}\to0$.

For $i=j\ne0$, after dividing the corresponding block equation by $\gamma$, its coefficient matrix converges to $N_i-\frac{3}{2}I_{p_i}$, which is Hurwitz.
Therefore, by continuity and uniqueness of the solution to the Lyapunov equation, $\mathscr D_{ii}$ converges as $\gamma\downarrow0$.
Denote its limit by $\mathcal D_{ii}$.
Dividing the corresponding block equation by $\gamma$ and letting $\gamma\downarrow0$ gives
\[
\left(N_i-\frac{3}{2}I_{p_i}\right)^{\mathrm H}\mathcal D_{ii}
+\mathcal D_{ii}\left(N_i-\frac{3}{2}I_{p_i}\right)
=-2\varXi_{ii}.
\]
Since $N_i-\frac{3}{2}I_{p_i}$ is Hurwitz and $\varXi_{ii}>0$, the Lyapunov equation theory gives $\mathcal D_{ii}>0$.

Now define $\mathscr Z=\gamma^{-1}\mathcal L^{-1}Z\mathcal L^{-1}$.
Since $\mathscr Z=\mathscr X-\mathscr D$, it follows that $\mathscr Z\to\mathcal Z=\bigoplus_{i=0}^{s}\mathcal Z_{ii}$, where $\mathcal Z_{00}=\varXi_{00}$ and $\mathcal Z_{ii}=\varXi_{ii}-\mathcal D_{ii}$, $i=1,\ldots,s$.

It remains to show that $\mathcal Z_{ii}>0$ for $i\ge1$.
The limiting $(i,i)$ block of the scaled equation for $Z$ gives
\[
\left(N_i-\frac{3}{2}I_{p_i}\right)^{\mathrm H}\mathcal Z_{ii}
+\mathcal Z_{ii}\left(N_i-\frac{3}{2}I_{p_i}\right)
=-\varXi_{ii}\mathcal B_i\mathcal B_i^{\mathrm H}\varXi_{ii}.
\]
On the other hand, the limiting equation for $\varXi_{ii}$ can be equivalently written as
\[
\left(N_i-\frac{1}{2}I_{p_i}\right)^{\mathrm H}\varXi_{ii}
+\varXi_{ii}\left(N_i-\frac{1}{2}I_{p_i}\right)
=-\varXi_{ii}\mathcal B_i\mathcal B_i^{\mathrm H}\varXi_{ii}.
\]
Since $\varXi_{ii}>0$, the latter Lyapunov equation implies that $(\mathcal B_i^{\mathrm H}\varXi_{ii}, N_i-\frac{1}{2}I_{p_i})$ is observable.
Observability is unchanged by a scalar shift of the state matrix;
therefore, $(\mathcal B_i^{\mathrm H}\varXi_{ii}, N_i-\frac{3}{2}I_{p_i})$ is also observable.
The preceding Lyapunov equation for $\mathcal Z_{ii}$ then gives $\mathcal Z_{ii}>0$.
Since $\mathcal D_{ii}>0$, we further obtain $0<\mathcal Z_{ii}<\varXi_{ii}$.

Define $\mathscr Y=\gamma^{-1}\mathcal L^{-1}Y\mathcal L^{-1}$.
Since $Z=XY^{-1}X$, we have $\mathscr Z=\mathscr X\mathscr Y^{-1}\mathscr X$.
Consequently, $\mathscr Y\to \varTheta=\varXi\mathcal Z^{-1}\varXi =\bigoplus_{i=0}^{s}\varTheta_{ii}>0$.
Since $\mathcal Z_{00}=\varXi_{00}$, we have $\varTheta_{00}=\varXi_{00}$.
For $i=1,\ldots,s$, the relation $0<\mathcal Z_{ii}<\varXi_{ii}$ implies $\mathcal Z_{ii}^{-1}>\varXi_{ii}^{-1}$, and hence $\varTheta_{ii} =\varXi_{ii}\mathcal Z_{ii}^{-1}\varXi_{ii} >\varXi_{ii}$.
This proves \eqref{eqmLy}.

Finally, we prove \eqref{eqmLp}.
Define $\mathscr P=\gamma^{-1}\mathscr L^{-1} P\mathscr L^{-1}=[\mathscr P_{ij}]$.
For $i=1,\ldots,s$, since $\mathscr L_i=\mathcal L_i$, \eqref{eqmLx} and \eqref{eqmLy} give $\mathscr P_{ii}\to \varPhi_{ii}=\varTheta_{ii}-\varXi_{ii}>0$.

For the stable block, $\mathscr X_{00}=X_{00}\to\varXi_{00}$ and $\mathscr Y_{00}=Y_{00}\to \varTheta_{00}=\varXi_{00}$, so $P_{00}=Y_{00}-X_{00}\to0$.
Since each entry of $P_{00}$ is a rational function of $\gamma$ and $P_{00}\to0$ as $\gamma\downarrow0$, its Laurent expansion at $\gamma=0$ starts from a positive integer power of $\gamma$.
Hence $P_{00}=O(\gamma)$, and $P_{00}/\gamma$ has a finite limit; denote it by $\varPhi_{00}$.

Using the definitions of $\mathscr X$, $\mathscr Y$ and $\mathscr P$, equation \eqref{ARE} becomes $ \mathscr L^{-1}J^{\mathrm H}\mathscr L\mathscr P +\mathscr P\mathscr LJ\mathscr L^{-1} -\gamma\mathscr P\mathscr LBB^{\mathrm H} \mathscr L\mathscr P  = -\gamma\mathscr L^{-1}\mathcal L (\mathscr X+\mathscr Y) \mathcal L\mathscr L^{-1} $.
Since $\mathscr P>0$ and all its diagonal blocks are bounded, its off-diagonal blocks are also bounded.
Hence every sequence $\gamma_k\downarrow0$ contains a subsequence along which $\mathscr P$ converges.
Let $\varPhi=[\varPhi_{ij}]$ be any such accumulation point.

Taking the $(0,0)$ block limit in the above equation gives $J_0^{\mathrm H}\varPhi_{00} +\varPhi_{00}J_0=-2\varXi_{00}$.
Since $J_0$ is Hurwitz and $\varXi_{00}>0$, the Lyapunov equation theory gives $\varPhi_{00}>0$.

For $i=0$, $j\ne0$, the limiting Sylvester equation is $J_0^{\mathrm H}\varPhi_{0j} +\lambda_j\varPhi_{0j}=0$, which implies $\varPhi_{0j}=0$ by spectral separation.
Similarly, $\varPhi_{i0}=0$.
For $i\ne j$ with $i,j\ne0$, the limiting equation gives $(\bar\lambda_i+\lambda_j)\varPhi_{ij}=0$.
Since $\bar\lambda_i=-\lambda_i$ and $\lambda_i\ne\lambda_j$, we have $\varPhi_{ij}=0$.

Therefore, every accumulation point of $\mathscr P$ is the same block-diagonal matrix $\varPhi=\bigoplus_{i=0}^{s}\varPhi_{ii}>0$.
Consequently, $\mathscr P\to\varPhi$, which proves \eqref{eqmLp}.
\end{proof}}
We first suppose that $\phi(A)=0$.
In this case, $ \varXi_{00} $, $ \varTheta_{00} $ and $ \varPhi_{00} $ are absent.
From (\ref{eqmLx}), (\ref{eqmLy}) and \eqref{eqmLp} we have
\[
    \lim_{\gamma \downarrow 0}\lambda_{\min}(YX^{-1})=\lambda_{\min}( \varTheta \varXi^{-1} )>1. \label{eqzb77}
\]
It follows from
\[
    \lambda_{\max} \left(\!\!\frac{\gamma ( Y\!\!+\!\!X ) ( Y\!\!-\!\!X ) ^{-1} }{\gamma-\phi(A)} \!\right) = 1\!\!+\!2\!\left( \lambda_{\min}\!\left( YX^{-1} \right)\!-\!\!1  \right) ^{-1},
\]
that (\ref{ksmal}) is proved in this case.

We next suppose that $\phi(A)\neq 0 $, which implies $\phi(A)<0$.
Direct computation shows that $\lim_{\gamma \downarrow 0}\gamma^{\frac{1}{2}}\mathcal{L}\mathscr{L}^{-1} = I_{p_{0}}  \oplus 0 $.
Thus, with the help of \eqref{eqmLx} and \eqref{eqmLp}, we can get
\begin{align*}
    &\lim_{\gamma \downarrow 0} \lambda_{\max}\left(\frac{\gamma}{\gamma-\phi(A)} ( Y+X ) ( Y-X ) ^{-1} \right) \\
    =& -2( \phi(A) )^{-1}\lim_{\gamma \downarrow 0}  \lambda_{\max}(  \gamma X P ^{-1} )\\
    \le & -2( \phi(A) )^{-1} \lambda_{\max}(\varXi)\lambda_{\max}(\varPhi_{00}^{-1}) < \infty,
\end{align*}
which proves (\ref{ksmal}) in this case.
{
Moreover, for any compact interval $\gamma\in[\gamma_1,\gamma_2]\subset(0,\infty)$, $X(\gamma)$ and $Y(\gamma)$ are continuous with respect to $\gamma$.
Since $P(\gamma)=Y(\gamma)-X(\gamma)>0$, the eigenvalue ratio considered above is continuous on $[\gamma_1,\gamma_2]$ and thus admits a finite maximum.
Combining the boundedness for $\gamma\rightarrow0$, $\gamma\rightarrow\infty$, and the intermediate compact interval, there exists a constant $k>0$ such that the inequality in Item 4 holds for all $\gamma>0$.
}

\textbf{\emph{Proof of Item 5}}:
Based on the fact that $\lim_{\gamma \to \infty} L/\gamma^{\mu_{\mathrm{c}}-1}$ is a nonzero matrix, according to \eqref{eqLp}, we have
\begin{equation*}
    \lim_{\gamma \to  \infty}\frac{\lambda_{\max}(P)}{\gamma^{2\mu_{\mathrm{c}}-1}} =\lim_{\gamma \to \infty}\lambda_{\max}\left( \frac{L}{\gamma^{\mu_{\mathrm{c}}-1}}L^{-1}PL^{-1} \frac{L}{\gamma^{\mu_{\mathrm{c}}-1}}\right),
\end{equation*}
which is clearly a positive constant.
Similarly, since $ \lim_{\gamma \downarrow 0}\mathscr{L} $ is a nonzero matrix, it follows from \eqref{eqmLp} that $\lim_{\gamma \downarrow 0} \lambda_{\max}(P)/\gamma= \lim_{\gamma \downarrow 0} \lambda_{\max}( \mathscr{L} \mathscr{L}^{-1}P\mathscr{L}^{-1}\mathscr{L}/\gamma)$, which is clearly a positive constant.

\textbf{\emph{Proof of Item 6}}:
By taking derivative of \eqref{ARE} with respect to $\gamma$, we can get
\begin{align*}
    &A^{\mathrm{T}}\frac{\mathrm{d}P }{\mathrm{d}\gamma}+\frac{\mathrm{d}P }{\mathrm{d}\gamma}A-\frac{\mathrm{d}P }{\mathrm{d}\gamma}BB^{\mathrm{T}}P -PBB^{\mathrm{T}}\frac{\mathrm{d}P }{\mathrm{d}\gamma} +  \gamma  \frac{\mathrm{d}P }{\mathrm{d}\gamma} \\
    =& -X- \gamma  \frac{\mathrm{d}X}{\mathrm{d}\gamma}-Y- \gamma  \frac{\mathrm{d}Y}{\mathrm{d}\gamma} + \gamma  \frac{\mathrm{d}P }{\mathrm{d}\gamma}  \\
    =& -X-Y-2\gamma \frac{\mathrm{d}X}{\mathrm{d}\gamma},
\end{align*}
rearranging which gives
\begin{equation}
    A_{\mathrm{c}}^{\mathrm{T}}\frac{\mathrm{d}P }{\mathrm{d}\gamma} + \frac{\mathrm{d}P }{\mathrm{d}\gamma}A_{\mathrm{c}}= -X-Y-2\gamma \frac{\mathrm{d}X}{\mathrm{d}\gamma},\label{dP1}
\end{equation}
where $A_{\mathrm{c}}=A-BB^{\mathrm{T}}P+\frac{\gamma}{2}I_{n}$.
According to \eqref{ARE}, we have
\begin{equation}
A_{\mathrm{c}}^{\mathrm{T}} P+PA_{\mathrm{c}} = -PBB^{\mathrm{T}}P-2\gamma X. \label{dP2}
\end{equation}
To proceed, we recall that the solution $X$ to the PLE (\ref{lyap0}) satisfies \cite{zhou20auto1}
\begin{equation}
\frac{\mathrm{d}X}{\mathrm{d}\gamma} \geq \frac{X}{n\gamma -2\mathrm{tr}(A)}. \label{eqzb9}
\end{equation}
Combining \eqref{dP1}, \eqref{dP2} and (\ref{eqzb9}) gives
\begin{align}
    &A_{\mathrm{c}}^{\mathrm{T}}\left( 2n\gamma\frac{\mathrm{d}P}{\mathrm{d}\gamma}-P \right)+\left( 2n\gamma\frac{\mathrm{d}P}{\mathrm{d}\gamma}-P \right)A_{\mathrm{c}} \notag \\
    =& -2n\gamma X-2n\gamma Y-4n\gamma^{2}\frac{\mathrm{d}X}{\mathrm{d}\gamma}+PBB^{\mathrm{T}}P+2\gamma X \notag \\
    \le & -2n\gamma X-2n\gamma Y-\frac{4n\gamma^{2}}{n\gamma-2\mathrm{tr}( A )} X \notag \\
    & +\mathrm{tr}( B^{\mathrm{T}}PB )(Y-X) + 2\gamma X \notag \\
    = &  \left( -4n\gamma+2\gamma-\frac{4n\gamma^{2}}{n\gamma-2\mathrm{tr}( A )}   \right) X <0, \label{eqzb8}
\end{align}
where we have used (\ref{eqdVleV0}) and noticed that $n\gamma-2\mathrm{tr}( A )>0$.
Based on \eqref{poles}, we have $\alpha( A_{\mathrm{c}})    = \alpha( A )-2\gamma + \gamma/2  =\alpha( A ) -3\gamma/2<0.$
Thus, by the Lyapunov equation theory, we know from (\ref{eqzb8}) that $ 2n\gamma \mathrm{d}P/\mathrm{d}\gamma-P>0$, which proves the inequality on the left hand side of (\ref{defdc}).

To prove the inequality on the right hand side of (\ref{defdc}), we first recall the well-known Jacobi's formula \cite{magnus2019matrix} : If  $ M $  is a differentiable map from the real numbers to  $ n\times n $  matrices, then
   \begin{equation}
       \frac{\mathrm{d}}{\mathrm{d}t}\det\left( M(t) \right) =\mathrm{tr}\left( \mathrm{adj}( M(t) ) \frac{\mathrm{d}M(t)}{\mathrm{d}t} \right). \label{eqdVleV2}
   \end{equation}
With the help of the Jacobi's formula (\ref{eqdVleV2}), we can get
\begin{align}
    \mathrm{tr}\left( 2n\gamma \frac{\mathrm{d}P}{\mathrm{d}\gamma}P^{-1} \right)
    =& 2n\gamma\mathrm{tr}\left( \frac{\mathrm{d}P}{\mathrm{d}\gamma}\frac{\mathrm{adj}(P)}{\det(P)} \right) \notag\\
    =& \frac{2n\gamma}{\det(P)}\frac{\mathrm{d}}{\mathrm{d}\gamma}( \det(P) ) \notag \\
    =& 2n\gamma \frac{\mathrm{d}}{\mathrm{d}\gamma}( \ln( \det( P ))). \label{eqzb1215}
\end{align}
Since $P$ is a rational matrix of  $ \gamma $, $ \det(P) $  is also a rational function of  $ \gamma $.
According to fundamental theorem of algebra, $ \det(P) $ can be written as follows
\begin{equation*}
    \det(P)=\frac{q_{0}\gamma^{n_{0}}\prod_{i=1}^{n_{1}}( \gamma-q_{i} )}{\prod_{j=1}^{n_{2}}( \gamma-r_{j} )} \triangleq\frac{Q(\gamma)}{R(\gamma)},
\end{equation*}
where $ q_{0} $ is a constant, $ n_{0}>0,\ n_{1} \geq 0 $ and $ n_{2} \geq 0 $ are integers,  $ \{ q_{i} \} $  and  $ \{ r_{j} \} $  are the nonzero roots of the polynomials $ Q(\gamma) $  and  $ R(\gamma) $, respectively.
Here $ n_{0}>0$ follows from Item 5 of Theorem \ref{theoP} that $\lim _{\gamma \downarrow 0} P =0$.
Since $ P $  and  $ P^{-1} $  are positive definite matrices for any  $ \gamma>0 $,  $ \det(P) $  and  $ \det(P^{-1}) $  are positive functions, which means that $ Q(\gamma) $  and  $ R(\gamma) $ have no zeros on the interval  $ (0,\infty) $, namely, $ \{ q_{i} \} $  and  $ \{ r_{j} \} $ are not positive real numbers.
{
Since \(\det(P(\gamma))>0\) for all \(\gamma>0\), its real logarithm is well defined.
Using the logarithmic derivative of the rational function \(\det(P(\gamma))\), we obtain
\begin{align}
  \gamma\frac{\mathrm{d}}{\mathrm{d}\gamma}\ln(\det(P(\gamma)))
&=\gamma\frac{\dfrac{\mathrm{d}}{\mathrm{d}\gamma}\det(P(\gamma))}
{\det(P(\gamma))}\notag\\
&=n_0+\sum_{i=1}^{n_1}\frac{\gamma}{\gamma-q_i}
-\sum_{j=1}^{n_2}\frac{\gamma}{\gamma-r_j}.
    \label{eqzb1215b}
\end{align}
Although nonreal roots may occur, they appear in conjugate pairs, and hence the expression in \eqref{eqzb1215b} is real for every \(\gamma>0\).
}
When $ \gamma \to \infty $, we have
\begin{equation*}
    \lim_{\gamma \to \infty} \left( n_{0} +\sum_{i=1}^{n_{1}} \frac{\gamma}{\gamma-q_{i}} - \sum_{j=1}^{n_{2}} \frac{\gamma}{\gamma-r_{j}} \right)=n_{0}+n_{1}-n_{2},
\end{equation*}
and when $ \gamma \downarrow 0 $, we can get
\begin{equation*}
    \lim_{\gamma \downarrow 0} \left( n_{0}+\sum_{i=1}^{n_{1}} \frac{\gamma}{\gamma-q_{i}} - \sum_{j=1}^{n_{2}} \frac{\gamma}{\gamma-r_{j}} \right)=n_{0}.
\end{equation*}
By the property of continuous functions, it can be obtained from the left hand side of \eqref{defdc}, (\ref{eqzb1215}) and (\ref{eqzb1215b}) that
\begin{align}
    \delta_{\mathrm{c}}&\triangleq \sup_{\gamma>0} \left\{\lambda_{\max}\left( 2n\gamma \frac{\mathrm{d}P}{\mathrm{d}\gamma}P^{-1} \right)  \right\}
    \label{caldc}\\
    &\leq \sup_{\gamma>0}\left\{ \mathrm{tr}\left( 2n\gamma \frac{\mathrm{d}P}{\mathrm{d}\gamma}P^{-1} \right) \right\} \notag \\
   &=2n\left( n_{0} +\sup_{\gamma>0}\left\{\sum_{i=1}^{n_{1}} \frac{\gamma}{\gamma-q_{i}} - \sum_{j=1}^{n_{2}} \frac{\gamma}{\gamma-r_{j}}\right\} \right) \notag \\
   &< \infty, \notag
\end{align}
from which it follows that
\begin{align*}
    2n\gamma P^{-\frac{1}{2}}\frac{\mathrm{d}P}{\mathrm{d}\gamma}P^{-\frac{1}{2}}
    \le& \lambda_{\max}\left( 2n\gamma P^{-\frac{1}{2}}\frac{\mathrm{d}P}{\mathrm{d}\gamma}P^{-\frac{1}{2}} \right) I_{n} \notag\\
    =& \lambda_{\max}\left( 2n\gamma \frac{\mathrm{d}P}{\mathrm{d}\gamma}P^{-1} \right) I_{n}\notag \\
    \leq & \sup_{\gamma>0} \left\{\lambda_{\max}\left( 2n\gamma \frac{\mathrm{d}P}{\mathrm{d}\gamma}P^{-1} \right)  \right\}  I_{n} \\
   =& \delta_{\mathrm{c}}  I_{n}   \leq  \infty,
\end{align*}
which is just the inequality on the right hand side of (\ref{defdc}).

{
We next refine the above uniform estimate by constructing the pointwise upper bound $f(\gamma)$ in Item 6 of Theorem~\ref{theoP}.
Define $H(\gamma)=2n\gamma P^{-\frac{1}{2}} (\mathrm{d}P/\mathrm{d}\gamma)P^{-\frac{1}{2}}$.
It follows from the left hand side of (\ref{defdc}) that $H(\gamma)>I_n$, and hence $H(\gamma)$ is symmetric positive definite.
Let $\lambda_1,\ldots,\lambda_n$ be the eigenvalues of $H(\gamma)$, with $\lambda_n=\lambda_{\max}(H(\gamma))$, and denote their mean by $\bar{\lambda}=\mathrm{tr}(H(\gamma))/n$.
Since $\sum_{i=1}^{n-1}(\lambda_i-\bar{\lambda}) =-(\lambda_n-\bar{\lambda})$, the Cauchy--Schwarz inequality gives $ \sum_{i=1}^{n-1}(\lambda_i-\bar{\lambda})^2 \geq \frac{(\lambda_n-\bar{\lambda})^2}{n-1} $.
Consequently, $\sum_{i=1}^{n}(\lambda_i-\bar{\lambda})^2 \geq n(\lambda_n-\bar{\lambda})^2/(n-1)$.
Since $\sum_{i=1}^{n}(\lambda_i-\bar{\lambda})^2 =\mathrm{tr}(H^2(\gamma))-(\mathrm{tr}(H(\gamma)))^2/n$, it follows that
\begin{align*}
    \lambda_{\max}(H(\gamma))
    \leq&
    \frac{1}{n}\mathrm{tr}\big(H\big) \!+\!
    \sqrt{ \frac{n-1}{n} \left( \mathrm{tr}\big(H^{2}\big) - \frac{1}{n} \big(\mathrm{tr}(H)\big)^{2}
    \right)}.
\end{align*}

By the definition of $f(\gamma)$ in Item 6 of Theorem~\ref{theoP}, the above inequality and (\ref{caldc}) imply $\lambda_{\max}(H(\gamma)) \leq f(\gamma)\leq\delta_{\mathrm{c}}$ for all $\gamma>0$.
Since $H(\gamma)$ is symmetric, we therefore have $ H(\gamma) \leq f(\gamma)I_n \leq \delta_{\mathrm{c}}I_n $.
Recalling the definition of $H(\gamma)$ and applying a congruence transformation with $P^{1/2}$ yield $ 2n\gamma\frac{\mathrm{d}P}{\mathrm{d}\gamma} \leq f(\gamma)P \leq \delta_{\mathrm{c}}P$.
Therefore, $ \frac{\mathrm{d}P}{\mathrm{d}\gamma} \leq \frac{f(\gamma)P}{2n\gamma} \leq \frac{\delta_{\mathrm{c}}P}{2n\gamma}$, which completes the proof of the refined inequalities in (\ref{defdc}).
}
{
\subsection*{C: The Controller Parameters Designed in \cite{song2017time}}
For the considered fourth-order system, the unconstrained control command in \cite{song2017time} is given by $ U_{2}=2x_{3}+2x_{4}-(L_{0}+L_{1}+kz)/\mu, $ where $ z=w_{4}+k_{1}w_{1}+k_{2}w_{2}+k_{3}w_{3} $, $ w_{1}=\mu x_{1} $, $ w_{2}=\mu(a_{1}x_{1}+x_{2}) $, $ w_{3}=\mu(a_{2}x_{1}+2a_{1}x_{2}+x_{3}) $, $ w_{4}=\mu(a_{3}x_{1}+3a_{2}x_{2}+3a_{1}x_{3}+x_{4}) $, $ L_{0}=\mu(a_{4}x_{1}+4a_{3}x_{2}+6a_{2}x_{3}+4a_{1}x_{4}) $, $ L_{1}=k_{1}w_{2}+k_{2}w_{3}+k_{3}w_{4} $, $ \mu=T^{n+p}/(T-t)^{n+p} $,
% \[
% \begin{aligned}
% a_{1}&=\frac{n+m}{T}\frac{T}{T-t},\\
% a_{2}&=\frac{(n+m)(n+m+1)}{T^{2}}
% \left(\frac{T}{T-t}\right)^2,\\
% a_{3}&=\frac{(n+m)(n+m+1)(n+m+2)}{T^{3}}
% \left(\frac{T}{T-t}\right)^3,\\
% a_{4}&=\frac{(n+m)(n+m+1)(n+m+2)(n+m+3)}{T^{4}}
% \left(\frac{T}{T-t}\right)^4 .
% \end{aligned}
% \]
$ a_{1}=(n+p)/(T-t) $, $ a_{2}=(n+p)(n+p+1)/(T-t)^{2} $, $ a_{3}=(n+p)(n+p+1)(n+p+2)/(T-t)^{3} $ and $ a_{4}=(n+p)(n+p+1)(n+p+2)(n+p+3)/(T-t)^{4} $.
The parameters of the compared controller are selected as $ n=4,\ p=2,\ k=2 $, and $ [k_{1},k_{2},k_{3}]^{\mathrm{T}}=[8,12,6]^{\mathrm{T}} $.
% The prescribed settling times and initial conditions are selected as $ T\in\{64,108,150,192\} $ and $ x(0)=\frac{D}{2}[1,-1,1,-1]^{\mathrm{T}}, \ D\in\{1,2,3,4\} $.
}

\end{document}